\documentclass[reqno, 10pt]{amsart}
\usepackage{amssymb, amsmath}
\usepackage{amsthm, amsfonts,mathrsfs}
\usepackage{times}

\def\inte#1{
\displaystyle\mathop{#1\kern0pt}^\circ }

\def\virgp{\raise 2pt\hbox{,}}
\def\cdotpv{\raise 2pt\hbox{;}}

\def\C{\mathop{\mathbb C\kern 0pt}\nolimits}
\def\DD{\mathop{\mathbb D\kern 0pt}\nolimits}
\def\EE{\mathop{{\mathbb E \kern 0pt}}\nolimits}
\def\K{\mathop{\mathbb K\kern 0pt}\nolimits}
\def\N{\mathop{\mathbb N\kern 0pt}\nolimits}
\def\Q{\mathop{\mathbb Q\kern 0pt}\nolimits}
\def\R{\mathop{\mathbb R\kern 0pt}\nolimits}
\def\SS{\mathop{\mathbb S\kern 0pt}\nolimits}
\def\ZZ{\mathop{\mathbb Z\kern 0pt}\nolimits}
\def\TT{\mathop{\mathbb T\kern 0pt}\nolimits}
\def\P{\mathop{\mathbb P\kern 0pt}\nolimits}

\newcommand{\beq}{\begin{equation}}
\newcommand{\eeq}{\end{equation}}
\newcommand{\ben}{\begin{eqnarray}}
\newcommand{\een}{\end{eqnarray}}
\newcommand{\beno}{\begin{eqnarray*}}
\newcommand{\eeno}{\end{eqnarray*}}

\newtheorem{defi}{Definition}[section]
\newtheorem{thm}{Theorem}[section]
\newtheorem{lem}{Lemma}[section]
\newtheorem{rmk}{Remark}[section]

\newtheorem{prop}{Proposition}[section]
\renewcommand{\theequation}{\thesection.\arabic{equation}}

\begin{document}

\title[On the wave-breaking phenomena and global existence for $\mu$-CH equation]
{On the wave-breaking phenomena and global existence for the
generalized periodic Camassa-Holm equation}

\author[Guilong Gui]{Guilong Gui}
\address[Guilong Gui]{Department of Mathematics, Jiangsu University, Zhenjiang 212013, P. R.
China\\ and The Institute of Mathematical Sciences, The Chinese
University of Hong Kong, Shatin, Hong Kong} \email{glgui@amss.ac.cn}
\author[Yue Liu]{Yue Liu}
\address[Yue Liu]{Department of Mathematics, University of Texas, Arlington, TX
76019-0408, USA} \email{yliu@uta.edu}
\author[Min Zhu]{Min Zhu}
\address[Min Zhu]{Department of Mathematics, Nanjing Forestry
University, Nanjing 210037, P. R. China \\
and Department of Mathematics, Southeast University, Nanjing 210097,
P. R. China}\email{zhumin@njfu.edu.cn}

\maketitle
\begin{abstract}
Considered herein is the initial-value problem for the generalized periodic Camassa-Holm equation which is related
to the Camassa-Holm equation and the Hunter-Saxton
equation. Sufficient conditions guaranteeing the development of
breaking waves in finite time are demonstrated. On the other hand,
the existence of strong permanent waves is established with certain
 initial profiles depending on the linear dispersive parameter
in a range of the Sobolev spaces. Moreover, the admissible global weak
solution in the energy space is obtained.

\date{}

\end{abstract}


\vskip 0.2cm


\renewcommand{\theequation}{\thesection.\arabic{equation}}
\setcounter{equation}{0}

\section{Introduction}
We study here the initial-value problem associated with the
generalized periodic Camassa-Holm ($\mu$-CH) equation \cite{khe}, namely,
\begin{equation}\label{e1.1}
\left\{
 \begin{array}{ll}
\begin{split}
&\mu(u_t)-u_{xxt}+2\mu(u)u_x+2\kappa u_x=2u_xu_{xx}+uu_{xxx}, \quad t > 0,  \quad x \in \mathbb{R}, \\
&u(0,x)=u_0(x), \qquad x \in \mathbb{R}, \\
&u(t,x+1)=u(t,x),\quad\quad\quad t \ge 0, \quad x\in\mathbb{R},
 \end{split}
\end{array} \right.
\end{equation}
where $u(t,x)$ is a time-dependent function on the unit circle
$\mathbb{S}=\mathbb{R}/\mathbb{Z}$ and
$\mu(u)=\int_{\mathbb{S}}u(t,x)dx$ denotes its mean, the parameter
$\kappa \in \mathbb{R}$. Obviously, if $\mu(u)=0$, which implies
that $ \mu(u_t) = 0, $  then this equation reduces to the
Hunter-Saxton (HS) equation \cite{hun1}, which is also a short wave
limit of the Camassa-Holm (CH) equation \cite{acdm, cam, conl2,
fuc}. Equivalently, the initial value problem \eqref{e1.1} can be
rewritten as the following mixed hyperbolic-elliptic type system.
\begin{equation}\label{e1.4}
 \begin{cases}
u_t+uu_x+\partial_x P=0, \qquad t > 0, \quad x \in \mathbb{R}, \\
(\mu-\partial_x^2)P=2\mu(u)u+\frac{1}{2}u^2_x+2\kappa \, u , \quad t > 0, \quad x\in\mathbb{R}\\
u(t,x+1)=u(t,x),\quad\quad\quad t \ge 0, \quad x\in\mathbb{R},\\
u(0,x)=u_0(x),  \qquad \qquad \qquad x \in \mathbb{R}.
\end{cases}
\end{equation}
With $m=(\mu-\partial_x^2)u$, the first equation in \eqref{e1.1} may
be read as
\begin{equation}\label{e1.4-1}
m_t+u m_x+2 m u_x+2\kappa u_x=0.
\end{equation}

It is known that the Camassa-Holm equation is one of the following
family of equations with the parameter $\lambda=2$
\begin{equation}\label{f1.1}
m_t + u m_x + \lambda u_x m+2\kappa u_x = 0,
\end{equation}
with $ m = Au $ and $ A = 1 - \partial_x^2$, the parameter $\kappa
\in \mathbb{R}$. The family of equations are believed to be
integrable \cite{cam, deg2} only for $\lambda=2$ and $\lambda=3$.

It is observed that the $\mu$-CH equation is the corresponding
$\mu$-version of the family given by \eqref{f1.1} with $ m = Au$, $
A = \mu - \partial_x^2$, and the parameter $\lambda=2$.

It is clear that the closest relatives of the $\mu$-CH equation are the Camassa-Holm equation
with $ A = 1 - \partial_x^2 $
\begin{equation*}\label{1.2}
u_t-u_{txx}+3uu_x+2\kappa u_x=2u_xu_{xx}+uu_{xxx},
\end{equation*}
and the equation with $ A = - \partial_x^2 $
 \begin{equation}\label{1.3}
-u_{txx}+2\kappa u_x=2u_xu_{xx}+uu_{xxx}.
\end{equation}
It is noted that when $\kappa=0$, \eqref{1.3} becomes the
Hunter-Saxton equation, while if $\kappa \neq 0$, \eqref{1.3} is a
short wave limit of the Camassa-Holm (CH) equation, which is an
equation in the Dym hierarchy and admits a new class of soliton
solutions(called umbilic solitons) \cite{acdm}.

Both of the CH equation and the HS equation have attracted a lot of
attention among the integrable systems and the PDE communities. The
Camassa-Holm equation was introduced  in \cite{cam} as a shallow
water approximation and has a bi-Hamiltonian structure \cite{fuc},
whose relevance for water waves was established in \cite{conl2}. The
Hunter-Saxton equation firstly appeared in \cite{hun1} as an
asymptotic equation for rotators in liquid crystals. Recently, it
was claimed in \cite{La} that the equation might be relevant to the
modeling of tsunami, also see the discussion in \cite{CoJo}.

The Camassa-Holm equation is a completely integrable system with a
bi-Hamiltonian structure and hence it possesses an infinite sequence
of conservation laws \cite{cam, fuc}, see \cite{con6} for the
periodic case. When $\kappa=0$, it admits soliton-like solutions
(called peakons) in both periodic and non-periodic setting
\cite{cam} and the multi-soliton or infinite-soliton solutions
consisting of a train of peaked solitary waves or `peakons'
\cite{cam, CaHoTi}. These peakons  are weak solutions in the
distributional sense and shown to be stable \cite {CaHoTi, con-m,
con-s, Di-Mo, Di-Mo2}. The Camassa-Holm equation describes geodesic
flows on the infinite dimensional group ${\mathcal D}^s(\mathbb{S})$
of orientation-preserving diffeomorphisms of the unit circle
$\mathbb{S}$ of Sobolev class $H^s$ and endowed with a
right-invariant metric by the $H^1$ inner product \cite{kou, mis2}.
The Hunter-Saxton equation also describes the geodesic flow on the
homogeneous space of the group $\mathcal{D}^s(\mathbb{S})$ modulo
the subgroup of rigid rotations $Rot(\mathbb{S})\simeq\mathbb{S}$
equipped with the $\dot{H}^1$ right-invariant metric \cite{len1} at
the identity
\begin{equation*}\langle
u,v\rangle_{\dot{H}^1}=\int_{\mathbb{S}}u_xv_x dx.
\end{equation*}
The Hunter-Saxton  equation possesses a bi-Hamiltonian structure and
is formally integrable \cite{hun2}.

Another remarkable property of the Camassa-Holm equation is the
presence of breaking waves (i.e. the solution remains bounded while
its slope becomes unbounded in finite time \cite{wh}) \cite {cam,
con1, con2, con4, con6,  mis1}. Wave breaking is one of the most
intriguing long-standing problems of water wave theory \cite{wh}. It
is worth pointing out that Bressan and Constantin proved that the
solutions to the Camassa-Holm equation can be uniquely continued
after wave-breaking as either global conservative or global
dissipative weak solution in \cite{BrCo1} and \cite{BrCo2},
respectively. It is noted that  Xin and Zhang obtained the existence
of a global-in-time weak solution to the Camassa-Holm equation in
the energy space \cite{xz}, where authors basically follow the
approach in \cite{zhangz-4} to study the viscous approximate
solutions to the Camassa-Holm equation.

The $\mu$-CH was introduced by
Khesin, Lenells and Misiolek \cite{khe} (also called $\mu$-HS
equation). Similar to the HS equation \cite{hun1}, the $\mu$-CH
equation describes the propagation of weakly nonlinear orientation
waves in a massive nematic liquid crystal  with external magnetic
filed and self-interaction. Here, the solution $ u(t, x) $ of the
$\mu$-CH equation represents the director field of a nematic liquid crystal,
$ x $ is a space variable in a reference frame moving with the
linearized wave velocity, and $ t $ is a slow time variable. Nematic
liquid crystals are fields consisting of long rigid molecules.  The
$\mu$-CH equation is also an Euler equation on $\mathcal{D}^s(\mathbb{S})$
(the set of circle diffeomorphism of the Sobolev class $ H^s$) and
it describes the geodesic flow on $\mathcal{D}^s(\mathbb{S})$  with
the right-invariant metric given at the identity by the inner
product \cite{khe}
\begin{equation*}\langle u, v\rangle=\mu(u)\mu(v)+\int_{\mathbb{S}}u_xv_x d x.
\end{equation*}
It was shown in \cite{khe} that the $\mu$-CH equation is formally
integrable and can be viewed as the  compatibility condition between
\begin{equation*}
\psi_{xx}=\xi (m+\kappa) \psi \quad \mbox{and} \quad \psi_t=\left
(\frac1{2\xi}-u \right )\psi_x+\frac1{2}u_x\psi,
\end{equation*}
where $\xi \in\mathbb{C} $ is a spectral parameter and $m=\mu(u)-u_{xx}$.

On the other hand, the $\mu$-CH equation admits bi-Hamiltonian structure and infinite
hierarchy of conservation laws. The first few conservation laws in the hierarchy are
\begin{equation*}
H_0=\int_{\Bbb S}m\ d x, \quad H_1=\frac1{2}\int_{\Bbb S} mu \ d x,
\quad H_2=\int_{\Bbb S}\left (\mu(u)u^2+\kappa \,
u^2+\frac1{2}uu_x^2 \right ) d x.
\end{equation*}
It is noted that the Hunter-Saxton equation does not have any
bounded traveling-wave solutions at all, while the $\mu$-CH equation
admits traveling waves that can be regarded as the appropriate
candidates for solitons. It is shown in \cite{khe, len2} that when
$\kappa=0$, the $\mu$-CH equation admits not only periodic
one-peakon solution $ u(t, x) = \varphi(x-ct) $ where
\begin{equation*} \varphi(x) = \frac{c}{26} (12 x^2 + 23)
\end{equation*}
for $ x \in [ - \frac{1}{2}, \frac{1}{2} ] $ and $ \varphi $ is
extended periodically to the real line, but also the multi-peakons
of the form
\begin{equation*}u=\sum^N_{i=1}p_i(t)g(x-q_i(t)),\end{equation*}
where $g(x)=\frac1{2}x(x-1)+\frac{13}{12}$ is the Green function of
the operator $(\mu -\partial_x^2)^{-1}$.
\begin{rmk}\label{rmk-a-opera-1}
The operator $A=\mu-\partial_x^2$ is elliptic and an isomorphism
between $H^s(\mathbb{S})$ and $H^{s-2}(\mathbb{S})$ since
\begin{equation*}\label{fourier-1-1}
(\widehat{Au})(k)=
\begin{cases}
(1+k^2)\,\widehat{u}(k), \quad \mbox{for} \quad k=0,\\
k^2 \,\widehat{u}(k), \quad \mbox{for} \quad k \in
\mathbb{Z}\setminus \{0\},
\end{cases}
\end{equation*}
where we denote the Fourier transform of a function $f$ in the torus
$\mathbb{S}$ by $\hat{f}(k)$ with the frequency $k \in \mathbb{Z}$.
In particular, if $u$ is constant, then $A\,u=u=A^{-1}\,u$.
\end{rmk}
According to the Green function of the operator $A^{-1}=(\mu
-\partial_x^2)^{-1}$(that is, $g(x)=\frac1{2}x(x-1)+\frac{13}{12}$),
the inverse $v=A^{-1}w$ can be given explicitly by
\begin{equation}\label{formula-1.5}
\begin{split}
&v(x)= \left( \frac{x^2}{2}-\frac{x}{2}+\frac{13}{12} \right )
\mu(w)+ \left ( x-\frac1{2}\right )\int^1_0\int^y_0w(s)  \,ds dy \\
&\qquad\qquad-\int^x_0\int^s_0w(r)\,
drds+\int^1_0\int^y_0\int^s_0w(r)\, dr ds d
y \\
&= \left (\frac{x^2}{2}-\frac{x}{2}+\frac{13}{12} \right )
\mu(w)+ \left (x-\frac1{2} \right )\int^1_0\int^y_0w(s)  \,ds
dy+\int^1_0\int^y_x\int^s_0w(r)\, dr ds d y .
\end{split}
\end{equation}
Since $A^{-1}$ and $\partial_x$ commute, the following identities
hold
\begin{equation*}\label{formula-1.6}
A^{-1}\partial_xw(x)=\left (x-\frac1{2}\right )\int^1_0w(x)  d
x-\int^x_0w(y)  dy+\int^1_0\int^x_0w(y)  d y d x,
\end{equation*}
and
\begin{equation}\label{formula-1.7}
A^{-1}\partial^2_xw(x)=-w(x)+\int^1_0 w(x)  dx.
\end{equation}
Thanks to \eqref{formula-1.5}, we can read explicitly the
formulation of $P$ in \eqref{e1.4} as
\begin{equation*}\label{formula-2.1}
\begin{split}
P=& \left (x^2-x+\frac{13}{6} \right )(\mu(u)+\kappa)\mu(u)\, +(2x-1
)(\mu(u)+\kappa) \int^1_0\int^y_0
u(s)  \,ds dy \\
&+2(\mu(u)+\kappa) \int^1_0\int^y_x\int^s_0 u(r)\, dr ds d y
+\frac{1}{2} \int^1_0\int^y_x\int^s_0(\partial_x u)^2(r)\, dr ds d
y\\
&+\frac{1}{2} \left (\frac{x^2}{2}-\frac{x}{2}+\frac{13}{12} \right
) \|\partial_x u\|_{L^2}^2+ \frac{1}{2} \left (x-\frac1{2} \right
)\int^1_0\int^y_0(\partial_x u)^2(s)  \,ds dy,
\end{split}
\end{equation*}
which leads to
\begin{equation}\label{formula-2.2}
\begin{split}
&\partial_xP= \left (\frac{1}{2}x-\frac1{4} \right )\left
(2\mu(u)(\mu(u)+\kappa)+\|\partial_x u\|_{L^2}^2\right)+ \frac{1}{2}
\int^1_0\int^y_0(\partial_x u)^2(s) \,ds dy
\\
&\quad+2(\mu(u)+\kappa) \left (\int^1_0\int^x_0u(y)  \,d y d
x-\int^x_0u(y)\, dy \right )-\frac{1}{2} \int^x_0(\partial_x
u)^2(y)\, dy.
\end{split}
\end{equation}
Note that $ H^{s} \hookrightarrow Lip$ for $s>\frac{3}{2}$. From the
theory of the transport equation point of view, one may define a
strong solution to \eqref{e1.4} as follows.
\begin{defi}\label{def-gss-1-1}
If $u\in C([0,T), H^{s}(\mathbb{S}))\cap C^{1}([0,T),
H^{s-1}(\mathbb{S}))$ with $s>\frac{3}{2}$ satisfies \eqref{e1.4},
then $u$ is called a strong solution to \eqref{e1.4}. If $u$ is a
strong solution on $[0,T)$ for every $T>0$, then it is called global
strong solution to \eqref{e1.4}.
\end{defi}
One of our goals in this paper is concerned with the existence of a
global weak solution in the energy space $H^1$, which is motivated by the
work in \cite{xz}.
\begin{defi}\label{def-gws-1-1}
 A continuous function $u=u(t, x)$ is said to be an admissible global weak
solution to the initial-value problem \eqref{e1.4} if

(i) $u(t, x) \in C(\mathbb{R}^{+} \times \mathbb{S}) \cap
L^{\infty}(\mathbb{R}^{+}, H^1(\mathbb{S}))$ and
\begin{equation}\label{energy-ine-1}
\mu(u)=\mu(u_0) \quad \mbox{and} \quad \|\partial_x u(t,
\cdot)\|_{L^{2}(\mathbb{S})} \leq  \|\partial_x
u_0(\cdot)\|_{L^{2}(\mathbb{S})} \quad \forall \quad t>0;
\end{equation}
(ii) $u(t, x)$ satisfies the equations in \eqref{e1.4} in the sense
of distributions and takes on the initial data pointwise.

\end{defi}

Our main results of the present paper are Theorems \ref{t4.3}-
\ref{t4.6} (wave-breaking), Theorems
\ref{thm-gss-1}-\ref{thm-gss-2}(Global strong solution), and Theorem
\ref{thm-main-1}(Global weak solution).

The remainder of the paper is organized as follows. In Section 2,
some {\it a priori} estimates and basic properties on its strong
solutions to the $\mu$-CH equation are recalled and derived, which
are constantly used in the whole paper.  In Section 3, the results
of blow-up to strong solutions are established  in details. It is
shown that the solutions of the $\mu$-CH equation can only have
singularities which correspond to wave breaking (Theorems
\ref{t4.3}-\ref{t4.6}).  Two sufficient conditions for the existence
of global strong solutions (Theorems
\ref{thm-gss-1}-\ref{thm-gss-2}) are specified in Section 4. The
existence of an admissible global weak solution in the energy space
$H^1$ (Theorem \ref{thm-main-1}) is demostrated in the last section,
Section 5.

 \vskip 0.1cm

\noindent {\it Notations}. Throughout this paper, we identity all
spaces of periodic functions with function spaces over the unit
circle $\mathbb{S}$ in $\mathbb{R}^2$, i. e.
$\mathbb{S}=\mathbb{R}/\mathbb{Z}$. Since all space of functions are
over $\mathbb{S}$, for simplicity, we drop $\mathbb{S}$ in our
notations of function spaces if there is no ambiguity. For a given
Banach space $Z$, we denote its norm by $\|\cdot\|_Z$.

\setcounter{equation}{0}
\section{Preliminaries}

In the following, we establish some {\it a priori} estimates for the
$\mu$-CH equation. Recall that the first two conserved quantities of
the $\mu$-CH equation are
\begin{equation*}H_0=\int_{\Bbb S}m\ d
x=\int_{\Bbb S}\left (\mu(u)-u_{xx}\right )  d x=\mu(u(t)),
\end{equation*}
and \begin{equation*}H_1=\frac1{2}\int_{\Bbb S}mu \ d
x=\frac1{2}\mu^2(u(t))+\frac1{2}\int_{\Bbb S}u^2_x(t,x)  d
x.\end{equation*}
It is easy to see that  $\mu(u(t))$ and $
\int_{\Bbb S}u^2_x(t,x) d x $ are conserved in time \cite{khe}. Thus
\begin{equation*}\label{2.0}
\mu(u_t)=0.
 \end{equation*}

For the sake of convenience, let
\begin{equation}\label{2.1}
\mu_0=\mu(u_0) =\mu(u(t))  =\int_{\mathbb{S}} u(t, x)  d x
 \end{equation}
 and
\begin{equation}\label{2.2}
\mu_1=\left(\int_{\Bbb S}u^2_x(0,x)  d x
\right)^{\frac1{2}}=\left(\int_{\Bbb S}u^2_x(t,x)  d
x\right)^{\frac1{2}}.
\end{equation}
Then $\mu_0$ and $\mu_1$ are constants and independent of time $t$.
\begin{lem}\label{l2.1}\cite{con2}
If $f \in H^3(\mathbb{S}) $ is such that $\int_{\mathbb{S}}f(x)\
dx=a_0/2,$ then for every $\varepsilon>0$, we have
$$\max\limits_{x \in \mathbb{S}}f^2(x)\le\dfrac{\varepsilon+2}{24}\int_{\mathbb{S}}f^2_x(x)  d x
+\dfrac{\varepsilon+2}{4\varepsilon}a^2_0.$$
\end{lem}

\begin{rmk}\label{rmk-2.1}  Since $H^3$ is dense in $H^1,$
Lemma \ref{l2.1} also holds for every $f \in H^1(\mathbb{S})$.
Moreover, if $\int_\mathbb{S}f(x)\ dx=0$, from the
 deduction of this lemma we arrive at the following inequality
\begin{equation}\label{e2.1}
\max\limits_{x\in
\mathbb{S}}f^2(x)\le\dfrac1{12}\int_{\mathbb{S}}f^2_x(x) dx,\quad
x\in\mathbb{S}, \quad f \in H^1(\mathbb{S}).
\end{equation}
\end{rmk}
\begin{lem}\label{l2.2}\cite{but}
For every $\  f(x)\in H^1(a,b)$ periodic and with zero average, i.e.
such that $\int^b_af(x)\ dx=0$, we have
\begin{equation*}\int^b_af^2(x)\ d x\le\left ( \dfrac{b-a}{2\pi}\right)^2\int^b_a|f'(x)|^2  dx,
\end{equation*}
and equality holds if and only if
\begin{equation*}f(x)=A\cos\left(\dfrac{2\pi x}{b-a}\right)+B\sin\left(\dfrac{2\pi x}{b-a}\right).
\end{equation*}
\end{lem}
Note that
\begin{equation*}\int_{\Bbb S}(u(t,x)-\mu_0)  d x=\mu_0-\mu_0=0.
\end{equation*}
By Lemma \ref{l2.1}, we find that
\begin{equation}\label{max-1-1}
\max\limits_{x\in \mathbb{S}}\;[u(t,x)-\mu_0]^2
\le\dfrac1{12}\int_{\mathbb{S}}u^2_x(t,x)  dx=\dfrac1{12}
\int_{\mathbb{S}}u^2_x(0,x)\ d x=\dfrac1{12}\mu^2_1.
\end{equation}
From the above estimate, we find that the amplitude of the wave
remains bounded in any time, that is,
\begin{equation*}\|u(t,\cdot)\|_{L^{\infty}}-|\mu_0|\le\|u(t,\cdot)-\mu_0\|_{L^{\infty}}\le\frac{\sqrt{3}}{6}\mu_1,
\end{equation*}
and so
\begin{equation}\label{e2.2}
\|u(t,\cdot)\|_{L^{\infty}}\le|\mu_0|+\frac{\sqrt{3}}{6}\mu_1.
\end{equation}
While thanks to Lemma \ref{l2.2}, we have
\begin{equation}\label{max-1-2}
\int_{\mathbb{S}}[u(t,x)-\mu_0]^2\, dx
\le\dfrac1{4\pi^2}\int_{\mathbb{S}}u^2_x(t,x) dx=\dfrac1{4\pi^2}
\int_{\mathbb{S}}u^2_x(0,x)\ d x=\dfrac1{4\pi^2}\mu^2_1.
\end{equation}
Therefore, one gets from \eqref{max-1-2} that
\begin{equation}\label{e2.3}
\begin{split}
\|u(t,x)\|^2_{L^2}&=\int_{\Bbb S}u^2(t,x)\, d x=\int_{\Bbb
S}[(u-\mu_0)^2+2\mu_0 u-\mu_0^2](t,x)\, d
x\\
&\leq\frac{1}{4\pi^2}\mu_1^2+\mu_0^2.
\end{split}
\end{equation}
It then follows that
\begin{equation*}\label{e2.4}
\|u(t,\cdot)\|_{H^1}^2=\int_{\Bbb S}u^2(t,x)\, d x +\int_{\Bbb
S}u^2_x(t,x)\, d x\leq\frac{1+4\pi^2}{4\pi^2}\mu_1^2+\mu_0^2.
\end{equation*}

Let us first state the following local well-posedness result of
\eqref{e1.4}, which was obtained in \cite{khe} and \cite{len2}
(up to a slight modification, the proof is omitted).
\begin{prop}\label{p2.1}
Let $u_0\in H^s(\mathbb{S})$, $s>3/2$. Then there exist a maximal
$T=T(u_{0})>0$  and a unique strong solution $u$ to \eqref{e1.4}
such that
\begin{equation*}
u=u(\cdot,u_{0})\in C([0,T), H^{s}(\mathbb{S}))\cap C^{1}([0,T),
H^{s-1}(\mathbb{S})).
\end{equation*}
Moreover, the solution depends continuously on the initial data,
i.e. the mapping $u_{0} \mapsto u(\cdot,u_{0}): H^{s}(\mathbb{S})
\rightarrow C([0,T),  H^{s}(\mathbb{S}))\cap C^{1}([0,T),
H^{s-1}(\mathbb{S}))$ is continuous.
\end{prop}
\begin{rmk}\label{rmk-index-1}
The maximal $T$ in Proposition \ref{p2.1} can be chosen independent
of $s$ in the following sense. If $u = u(\cdot, u_0) \in C\left([0,
T), H^s\right)\cap C^1\left([0, T),  H^{s-1}\right)$ to \eqref{e1.4}
and $u_0 \in H^{s'}$ for some $s' \neq s, \, s'
> \frac{3}{2}, $ then $u \in  C\left([0, T),  H^{s'}\right)\cap
C^1\left([0, T),  H^{s'-1}\right)$ and with the same $T$. In
particular, if $u_0 \in H^{\infty} = \bigcap_{s \geq 0}H^s,$ then $u
\in C\left([0, T),  H^{\infty}\right)$ (see \cite {flq} for the
details, or \cite{G-L-iumj, GR} for an adaptation of the Kato method
\cite{kat1} to the proof of this statement for the (generalized)
Camassa-Holm equation).
\end{rmk}

Let us now consider the following differential equation
\begin{equation}\label{diff-1-0}
\begin{cases}q_{t}=u(t,q),\quad &t\in[0,T),
\\ q(0,x) = x,&x\in \mathbb{R}.
\end{cases}
\end{equation}
Applying classical results in the theory of ordinary differential
equations, we have the following properties of  $ q $ which are
crucial in the proof of global existence.

\begin{lem}\label{lem-diff-1}
Let $u_{0}\in H^{s}(\mathbb{S}), \; s > \frac{3}{2}$, and let $T>0$
be the maximal existence time of the corresponding strong solution
$u$ to \eqref{e1.4}. Then  Eq.\eqref{diff-1-0} has a unique
solution $q \in C^{1}([0,T)\times \mathbb{R},\mathbb{R})$ such that
the map $q(t,\cdot)$ is an increasing diffeomorphism of $\mathbb{R}$
with
\begin{equation*}\label{diff-1-1}
q_{x}(t,x)=\exp \left ( \int_{0}^{t}u_{x}(s,q(s,x))ds \right ) >
0,\, \; \forall (t,x)\in [0,T)\times \mathbb{R}.
\end{equation*}
 Furthermore, setting
$m =\mu(u)-u_{xx}$, we have
\begin{equation*}\label{diff-1-2}
\left(m(t,q(t,x))+\kappa\right)q_{x}^{2}(t,x) =m_{0}(x)+\kappa,\quad
\forall (t,x)\in [0,T)\times\mathbb{R}.
\end{equation*}
\end{lem}

\begin{proof}
Since $u \in C^1 \left([0, T), H^{s-1}(\mathbb{S})\right)$ and
$H^{s}(\mathbb{S})\hookrightarrow C^1(\mathbb{{S}}),$ we see that
both functions $u(t, x)$ and $u_x(t, x)$ are bounded, Lipschitz in
the space variable $x$, and of class $C^1$ in time. Therefore, for
fixed $x \in \mathbb{R}$, \eqref{diff-1-0} is an ordinary
differential equation. Then well-known classical results in the
theory of ordinary differential equation yield that \eqref{diff-1-0}
has a unique solution $q(t, x) \in C^1 \left([0, T) \times
\mathbb{R}, \mathbb{R}\right).$

Differentiation of \eqref{diff-1-0} with respect to $x$ yields
\begin{equation}\label{diff-1-3}
\begin{cases}\frac{d}{dt}q_{x}=u_x(t,q)q_x,\quad t\in[0,T),
\\ q_x(0,x) = 1, \quad  x\in \mathbb{R}.
\end{cases}
\end{equation}
The solution to \eqref{diff-1-3} is given by
\begin{equation}\label{diff-1-4}
q_{x}(t,x)=\exp \left ( \int_{0}^{t}u_{x}(s,q(s,x))ds \right ),\, \;
(t,x)\in [0,T)\times \mathbb{R}.
\end{equation}
For every $T' < T,$ it follows from the Sobolev imbedding theorem
that

\begin{equation*}
\sup_{(s,x)\in [0,T')\times \mathbb{R}}|u_x(s, x)|< \infty.
\end{equation*} We infer from \eqref{diff-1-4}
that there exists a constant $K > 0$ such that $q_x(t, x) \geq
e^{-Kt}, \ (t,x)\in [0,T)\times \mathbb{R},$ which implies that the
map $q(t,\cdot)$ is an increasing diffeomorphism of $\mathbb{R}$
with
\begin{equation*}
q_{x}(t,x)=\exp \left ( \int_{0}^{t}u_{x}(s,q(s,x))ds \right ) > 0,\, \; \forall (t,x)\in
[0,T)\times \mathbb{R}.
\end{equation*}
On the other hand, combining \eqref{diff-1-3} with \eqref{e1.4-1},
we have
\begin{equation*}\label{diff-1-5}
\begin{split}
\frac{d}{dt}\left((m(t, q(t, x))+\kappa)q_x^{2}(t,
x)\right)&=\left(m_t+m_x
q_x \right) q_x^{2}(t, x)+2 (m+\kappa) q_x q_{xt}\\
&=q_x^{2}(m_t+m_x u+2 u_x m+2\kappa u_x)=0.
\end{split}
\end{equation*}
So,
\begin{equation*}\label{diff-1-6} \left(m(t,q(t,x))+\kappa\right)q_{x}^{2}(t,x) =m_{0}(x)+\kappa,
\quad \forall (t,x)\in [0,T)\times\mathbb{R}.
\end{equation*}
This completes the proof of Lemma \ref{lem-diff-1}.
\end{proof}

\begin{rmk}\label{rmk-diff-1}
Lemma \ref{lem-diff-1} shows that, if $m_0+\kappa = \mu(u_0)
-u_{0xx} +\kappa$ does not change sign, then $m\left( t \right)
+\kappa \,(\forall \ t)$ will not change sign, as long as $m\left( t
\right)$ exists.
\end{rmk}

\begin{rmk}\label{rmk-diff-2}
Since $q(t,\cdot):\mathbb{R}\rightarrow\mathbb{R}$ is a
diffeomorphism of the line for every $t\in[0,T)$, the
$L^{\infty}$-norm of any function $v(t,\cdot)\in L^{\infty}, \;
t\in[0,T)$ is preserved under the family of diffeomorphisms
$q(t,\cdot)$ with $t\in[0,T)$, that is,
\begin{equation*}
\|v(t,\cdot)\|_{L^{\infty}}=\|v(t,q(t,\cdot))\|_{L^{\infty}},\quad
t\in[0,T).
\end{equation*}
\end{rmk}

In \cite{khe}  and \cite{len2}, the authors also showed  that the
$\mu$-CH equation admits global (in time) solutions and  blow-up
solutions. It is our purpose here to derive  the precise wave-breaking scenarios
and determine the initial conditions guaranteeing the blow-up of strong solutions
to the initial-value problem  \eqref{e1.1}, which will significantly
improve the results in \cite{khe} and \cite{len2}.

As longs as the solution $u$ to \eqref{e1.4} is defined,
we set
 \begin{equation}\label{max-min-1}m_1(t) = \min_{x\in \mathbb{S}} [u_x(t, x)], \quad \mbox{and} \quad m_2(t) =
\max_{x\in \mathbb{S}} [u_x(t, x)]
 \end{equation}
 and further $x_1(t) \in
\mathbb{S}$ and $x_2(t)\in  \mathbb{S}$ are points where these
extrema are attained, i.e., $m_i(t) = u_x(t, x_i(t))$, $i = 1,\, 2$.
We will make use of the following lemma.

\begin{lem}\label{l2.3}\cite{con5}
Let $[0, T)$  be the maximal interval of existence of the solution
$u(t, x)$ of \eqref{e1.4} with the initial data $u_0 \in H^s$, $s >
\frac{3}{2}$, as given by Proposition \ref{p2.1}. Then the functions
$m_i(t)$, $i = 1,\, 2$, are absolutely continuous on $(0, T)$  with
 \begin{equation*}\frac{d m_{i}}{d t}=u_{xt}(t,x_{i}(t)),\quad  a.e. \quad on \quad(0,T).
 \end{equation*}
\end{lem}

\setcounter{equation}{0}

\section {Wave-breaking mechanism}

In this section, we  derive some sufficient conditions for the breaking waves to the initial-value
problem \eqref{e1.4}. We first state the precise wave-breaking
scenario for the problem \eqref{e1.4} in the following, which was
obtained in \cite{flq} (up to a slight modification).

\begin{prop}\label{t4.2}
Let $u_0\in H^s({\mathbb S}), s > 3/2$, and $u(t,x)$ be the solution
of the initial-value problem \eqref{e1.4} with life-span $T$. Then
$T$ is finite if and only if
 \begin{equation*}\underset{t\uparrow T}{\liminf}\left (\underset{x\in \mathbb{S}}{\inf}
\,u_x(t,x)\right )=-\infty.
\end{equation*}
\end{prop}
In what follows, we establish some sufficient conditions
guaranteeing the development of singularities by means of the
wave-breaking scenario. We are now in a position to give the first wave-breaking result for the
$\mu$-CH equation.

\begin{thm}\label{t4.3}
Let $u_0\in H^s(\mathbb{S}), s > 3/2$ and $T>0$ be the maximal time
of existence of the corresponding solution $u(t,x)$ to \eqref{e1.4}
with the initial data $u_0$. If $\
(\sqrt{3}/\pi)|\mu_0+\kappa|<\mu_1,$ where $\mu_0 $ and $ \mu_1$ are
defined in \eqref{2.1} and \eqref{2.2}, then the corresponding
solution $u(t,x)$ to \eqref{e1.4} associated with the $\mu$-CH
equation must blow up in finite time  $T$ with
 \begin{equation*}0<T\le\inf\limits_{\alpha \in I}\left (\dfrac{6}{1-6\alpha}+4\pi^2\alpha
\dfrac{1+|\int_{\mathbb{S}}u_{0x}^3(x)\;d x|}
{6\pi^2\alpha\mu_1^4-3(\mu_0+\kappa)^2\mu_1^2}\right )
 \end{equation*}
where $ I = \left
(\frac{(\mu_0+\kappa)^2}{2\pi^2\mu_1^2}, \, \frac1{6} \right )$,
such that
 \begin{equation*}\liminf\limits_{t\uparrow T}\left (\inf\limits_{x\in\mathbb{S}}u_x(t,x)\right )=-\infty.
  \end{equation*}
\end{thm}
\begin{proof}
Thanks to Remark \ref{rmk-index-1}, it suffices to consider the case
$s=3$. Differentiating the first equation in \eqref{e1.4} with
respect to $x$ yields
 \begin{equation}\label{one-deri-1}
 u_{tx}+u^2_x+uu_{xx}+A^{-1}\partial^2_x \left
(2u\mu_0+\frac1{2}u^2_x+2\kappa u \right )= 0.
\end{equation}
In view of \eqref{2.1}, \eqref{2.2} and \eqref{formula-1.7}, we have
 \begin{equation}\label{e4.0}
u_{tx}=-\frac1{2}u^2_x-uu_{xx}+2u(\mu_0+\kappa)-2\mu_0^2-\frac1{2}\mu_1^2-2\kappa
\mu_0.
\end{equation}
Multiplying (\ref {e4.0}) by $3u^2_x$ and integrating on
$\mathbb{S}$ with respect to $x$, we obtain for any $ t \in [0, T) $
that
 \begin{equation}\label{e3.1}
 \begin{split}
 \frac{d}{d t}\int_{\mathbb{S}}u_x^3\;d x &=\int_{\mathbb{S}}3u_x^2u_{xt}\;d x=-\frac{3}{2}\int_{\mathbb{S}}u_x^4\;d
 x-\int_{\mathbb{S}}3uu_x^2u_{xx}\;d x\\
 &\qquad\qquad\qquad+
 6(\mu_0+\kappa)\int_{\mathbb{S}}(u-\mu_0)u_x^2\;d x-\frac{3}{2}\left (\int_{\mathbb{S}}u_x^2\;d x\right )^2\\
 &=-\frac1{2}\int_{\mathbb{S}}u_x^4\;d x-\frac{3}{2}\mu_1^4
 +6(\mu_0+\kappa)\int_{\mathbb{S}}(u-\mu_0)u_x^2\;d x.
  \end{split}\end{equation}
On the other hand, it follows from Lemma \ref{l2.2} for any $ \alpha > 0 $ that
 \begin{equation*}
 \begin{split}
 (\mu_0+\kappa)\int_{\mathbb{S}}(u-\mu_0)u_x^2\;d x
 &\le|\mu_0+\kappa|\left (\int_{\mathbb{S}}(u-\mu_0)^2\;d x\right )^{\frac1{2}}
 \left (\int_{\mathbb{S}}u_x^4\;d x\right )^{\frac1{2}}\\
 &\le\frac{\alpha}{2}\int_{\mathbb{S}}u_x^4\;d x
 +\frac{(\mu_0+\kappa)^2}{2\alpha}\int_{\mathbb{S}}(u-\mu_0)^2\;d x\\
 &\le\frac{\alpha}{2}\int_{\mathbb{S}}u_x^4\;d x
 +\frac{(\mu_0+\kappa)^2}{8\pi^2\alpha}\int_{\mathbb{S}}u_x^2\;d x.
 \end{split}\end{equation*}
Therefore we deduce that
 \begin{equation}\label{blow-up-1-1}\dfrac{d}{d t}\int_{\mathbb{S}}u_x^3\;d x\le \left (3\alpha-\frac1{2}\right )
 \int_{\mathbb{S}}u_x^4
 \;d x-\frac{3}{2}\mu_1^4+\frac{3}{4\pi^2\alpha}(\mu_0+\kappa)^2\mu_1^2.
 \end{equation}
By the assumption of the theorem, we know that
 $ {(\mu_0+\kappa)^2}/({2\pi^2\mu_1^2})<1/6. $ Let   $ \alpha >0$ satisfy
 $$\frac{(\mu_0+\kappa)^2}{2\pi^2\mu_1^2}<\alpha<\frac1{6}.$$
This in turn implies that
$$ c_1:=\frac{1}{2}-3\alpha>0 \quad \mbox{and} \quad
c_2:=\frac{3}{2}\mu_1^4-\frac{3}{4\pi^2\alpha}(\mu_0+\kappa)^2\mu_1^2
>0.$$
Hence, applying H\"{o}lder's inequality to \eqref{blow-up-1-1} yields
\begin{equation*} \frac{d}{d t}\int_{\mathbb{S}}u_x^3\;d x\le
-c_1\int_{\mathbb{S}}u_x^4\;d x-c_2\le
 -c_1\left (\int_{\mathbb{S}}u_x^3\;d x\right )^{\frac{4}{3}}-c_2.
\end{equation*}
 Let  $V(t)=\int_{\mathbb{S}}u_x^3(t, x)\ d x$ with $ t \in [0, T).$ Then the above inequality can be rewritten as
\begin{equation}\label{ine-diff-3-1}
\frac{d}{d t}V(t)\le-c_1(V(t))^{\frac{4}{3}}-c_2\le-c_2<0,\quad
t\in[0,T).\end{equation}
This implies that $V(t)$ decreases strictly
in $[0, T). $  Let $t_1=(1+|V(0)|)/c_2.$ One can assume $ t_1 < T. $
Otherwise, $ T \le t_1 < \infty $ and the theorem is proved. It then
follows from \eqref{ine-diff-3-1} that
\begin{equation*}V(t)\le\left [\dfrac{3}{c_1(t-t_1)-3}\right ]^3\rightarrow-\infty, \quad\text{as}\quad t\rightarrow t_1+\frac{3}{c_1}.
\end{equation*}
On the other hand, we have
 \begin{equation*}V(t)=\int_{\mathbb{S}}u_x^3\;d x
 \ge\inf\limits_{x\in\mathbb{S}}u_x(t,x)\int_{\mathbb{S}}u_x^2\;d x
 =\mu_1^2\inf\limits_{x\in\mathbb{S}}u_x(t,x).
 \end{equation*}
 This then implies that $0<T \le t_1+3/c_1$ such that
\begin{equation*}\liminf\limits_{t\uparrow T}\left (\inf\limits_{x\in\mathbb{S}}u_x(t,x)\right )
=-\infty.\end{equation*}
 This completes the proof of Theorem \ref{t4.3}.
\end{proof}

In the case $ (\sqrt{3}/\pi)|\mu_0+\kappa|\ge\mu_1$, we have the
following wave-breaking result.
\begin{thm} \label{t4.5}
Let $u_0 \in H^s(\mathbb{S}), s
> 3/2$ and $T>0$ be the maximal time of existence of the
corresponding solution $u(t,x)$ to \eqref{e1.4} with the initial
data $u_0$. If  $ (\sqrt{3}/\pi)|\mu_0+\kappa|\ge\mu_1$ and
\begin{equation*}\inf_{x \in \mathbb{S}}  u'_0(x)
<-\sqrt{2\mu_1 \left
(\frac{\sqrt{3}}{3}|\mu_0+\kappa|-\frac1{2}\mu_1 \right )}:\equiv-K,
\end{equation*}
where $ u'_0(x) $ is the derivative of $ u_0(x) $ with respective to
$ x, $ then the corresponding solution $u(t,x)$ to \eqref{e1.4}
blows up in finite time  $T$ with
\begin{equation*}0<T\leq \frac{\inf_{x \in\mathbb{S}}
u'_0(x)}{K^2-(\inf_{x \in\mathbb{S}} u'_0(x))^2},
\end{equation*} such that
\begin{equation*}\liminf\limits_{t\uparrow T}\left (\inf\limits_{x\in\mathbb{S}}u_x(t,x)\right )
=-\infty.
\end{equation*}
\end{thm}
\begin{proof}
As discussed above, it suffices to consider the case $s=3$.
Note that the assumption $\  (\sqrt{3}/\pi)|\mu_0+\kappa|\ge\mu_1$
implies that $\  (2/\sqrt{3})|\mu_0+\kappa| > \mu_1$. Therefore the
non-negative constant $K$ is well-defined.

By Lemma \ref{l2.3}, there is $ x_0 \in \mathbb{S} $ such that $
\displaystyle u_0'(x_0) = \inf_{x \in \mathbb{S}} u_0'(x). $ Define
$w(t)=u_x(t,q(t,x_0))$, where $q(t,x_0)$ is the flow of
$u(t,q(t,x_0))$. Then
\begin{equation*}\dfrac{d}{d
t}w(t)=(u_{tx}+u_{xx}q_t)(t,q(t,x_0))=(u_{tx}+uu_{xx})(t,q(t,x_0)).
\end{equation*}
Substituting $(t,q(t,x_0))$ into \eqref{e4.0} and using
\eqref{e2.1}, we obtain
\begin{equation*}
\begin{split}
\dfrac{d}{d t}w(t)&=-\frac1{2}w^2(t)+2(\mu_0+\kappa)
u(t,q(t,x_0))-2\mu_0(\mu_0+\kappa)-\frac1{2}\mu_1^2\\
&=-\frac1{2}w^2(t)+2(\mu_0+\kappa)[u(t,q(t,x_0))-\mu_0]-\frac1{2}\mu^2_1,
\end{split}
\end{equation*}
which together with \eqref{max-1-1} implies that
\begin{equation}\label{4.32}
\begin{split}
\dfrac{d}{d t}w(t) &\leq
-\frac1{2}w^2(t)+\mu_1(\frac{\sqrt{3}}{3}|\mu_0+\kappa|-\frac1{2}\mu_1)
=-\frac1{2}w^2(t)+\frac{1}{2} K^2.
\end{split}
\end{equation}
By the assumption $w(0)= u_{0}'(x_0)< - K$, we have $w^2(0)> K^2$.
We now claim that $w(t)< - K$ holds for any $t \in [0, T).$ In fact,
assuming the contrary would, in view of $w(t)$ being continuous,
ensure the existence of $t_0 \in (0, T)$ such that $w^2(t)> K^2$ for
$t \in [0, t_0)$ but $w^2(t_0)= K^2$. Combining this with
\eqref{4.32} would give
\begin{equation}\label{4.33}
\frac{d}{dt}w(t)<0  \quad \mbox{a.e. } \, \mbox{on }\quad  [0, t_0).
\end{equation}
 Since $w(t)$ is absolutely continuous on $[0, t_0],$ an
 integration of this inequality would give the following inequality and
 we get the contradiction
 \begin{equation*}\label{4.34}w(t_0)<w(0)=u_{0}'(x_0)< - K.
 \end{equation*}
This proves the previous claim. Therefore, we get $\frac{d}{dt} w(t) <0$ on $[0, T)$, which
implies that $w(t)$ is  strictly decreasing  on $[0,T)$. Set
 \begin{equation*}\delta:=1-\left(\frac{K}{u'_0(x_0)}\right)^2 \in \left (0,\, 1\right ).
 \end{equation*}
And so
 \begin{equation*}\frac{K^2}{1-\delta}=(u'_0(x_0))^2 < w^2(t),
 \quad \mbox{i.e.} \quad K^2<(1-\delta)w^2(t).\end{equation*}
Therefore
 \begin{equation*}\dfrac{d}{d t}w(t)\le-\frac1{2}w^2(t)\left[1-(1-\delta)\right]= -\delta w^2(t),
  \quad t\in[0,T),
  \end{equation*}
which leads to
 \begin{equation*}w(t)\leq \frac{u'_0(x_0)}{1+\delta \,t \,u'_0(x_0)}\rightarrow-\infty,
 \quad\text{as}\quad t\rightarrow -\frac{1}{\delta \,u'_0(x_0)}.
 \end{equation*}
This implies
\begin{equation*}T\leq -\frac{1}{\delta \,u'_0(x_0)}=\frac{\inf_{x \in\mathbb{S}}
u'_0(x)}{K^2-(\inf_{x \in\mathbb{S}} u'_0(x))^2}<+\infty.
\end{equation*} In consequence, we have
\begin{equation*}\liminf\limits_{t\uparrow T}\left (\inf\limits_{x\in\mathbb{S}}u_x(t,x)\right )=-\infty.
\end{equation*}
This completes the proof of Theorem \ref{t4.5}.
\end{proof}

\begin{rmk}\label{rmk-blow-up-min} We can apply Lemma \ref{l2.3} to verify the above theorem
under the same conditions.
In fact, if we define
$w(t)=u_x(t,\xi(t))=\inf\limits_{x\in\mathbb{S}}[u_x(t,x)]$, then
for all $t\in[0,T)$, $u_{xx}(t,\xi(t))=0$. Thus if
$(\sqrt{3}/\pi)|\mu_0+\kappa|\ge\mu_1$, one finds that
$$
\frac{d}{dt} w(t)\le-\frac1{2}w^2(t)+\frac1{2}K^2,$$ where $K$ is
the same as  Theorem \ref{t4.5}. Then by means of the assumptions of
Theorem \ref{t4.5} and following the line of the proof of Theorem
\ref{t4.5}, we see that if
\begin{equation*}w(0)<-\sqrt{2\mu_1\left(\frac{\sqrt{3}}{3}|\mu_0+\kappa|-\frac1{2}\mu_1\right)},
\end{equation*} then $T$ is
finite and $\liminf\limits_{t\uparrow
T}\left(\inf\limits_{x\in\mathbb{S}}u_x(t,x)\right)=-\infty$.
\end{rmk}
Recall the definition of the extrema $m_1(t)$, $m_2(t)$ in
\eqref{max-min-1} of the slope $u_x(t, x)$ on the circle
$\mathbb{S}$, we may get the following wave-breaking result.
\begin{thm} \label{thm-blow-4}
Let $u_0 \in H^s(\mathbb{S}), s
> 3/2$ and $T>0$ be the maximal time of existence of the
corresponding solution $u(t,x)$ to \eqref{e1.4} with the initial
data $u_0$. If
 \begin{equation*} m_1(0)+m_2(0) < -8|\kappa| \quad \mbox{when} \quad \frac{2\sqrt{3}}{3}|\mu_0|<
 \mu_1, \quad \mbox{or}
\end{equation*}
 \begin{equation*} m_1(0)+m_2(0) < -8|\kappa|-2\sqrt{2}C_1 \quad \mbox{when} \quad
 \frac{2\sqrt{3}}{3}|\mu_0|\geq
 \mu_1
\end{equation*}
with
$C_1:=\sqrt{\left|\frac{\sqrt{3}}{3}|\mu_0|-\frac1{2}\mu_1\right|\mu_1}$,
then the corresponding solution $u(t,x)$ to \eqref{e1.4} blows up in
finite time  $T$.
\end{thm}
\begin{proof}
As discussed above, it suffices to consider the case $s=3$. In view of \eqref{one-deri-1},
 \eqref{2.1}, \eqref{2.2} and \eqref{formula-1.7}, together
with Remark \ref{rmk-a-opera-1} applied, we have
 \begin{equation}\label{blow-up-4-1-a}
 \begin{split}
u_{tx}&=-u^2_x-uu_{xx}-A^{-1}\partial^2_x (2u\mu_0+\frac1{2}u^2_x)-
2 \kappa A^{-1}\partial^2_x \,u
\\
&=-\frac1{2}u^2_x-uu_{xx}+2\mu_0(u-\mu_0)-\frac1{2}\mu_1^2-2\kappa
A^{-1}\partial^2_x \,  u.
\end{split}
\end{equation}
Thanks to \eqref{formula-1.5}, we obtain that
\begin{equation*}
\begin{split}
A^{-1}\partial^2_x   u &=(\frac{x^2}{2}-\frac{x}{2}+\frac{13}{12})
\mu(\partial^2_x u)+(x-\frac1{2})\int^1_0\int^y_0\partial^2_x u(s)
\,ds dy \\
&\qquad  +\int^1_0\int^y_x\int^s_0
\partial^2_x u(r)\, dr ds d y,
\end{split}
\end{equation*}
which implies
\begin{equation*}\label{blow-up-4-1-b}
\begin{split}
&|A^{-1}\partial^2_x   u |=\left|\frac{2x-1}{2}
\int^1_0\int^y_0\partial^2_x u(s) \,ds dy+\int^1_0\int^y_x\int^s_0
\partial^2_x u(r)\, dr ds d y\right|\\
&= \left|(x-\frac1{2} )\int^1_0(\partial_x u(y)-\partial_x u(0)) \,
dy+\int^1_0\int^y_x(\partial_x u(s)-\partial_x u(0))\, ds d y\right|
\\&\leq (m_2-m_1)\left(|x-\frac1{2}|+\int^1_0|y-x| d y\right)\\
&\leq (m_2-m_1)\left(|x-\frac1{2}|+x^2-x+\frac{1}{2}\right)\leq
m_2-m_1.
\end{split}
\end{equation*}
From this, together with \eqref{blow-up-4-1-a}, \eqref{max-1-1}, the
fact $u_{xx}(t, x_i(t))=0$ for a.e. $t \in [0, T)$, and Lemma
\ref{l2.3} applied, we deduce that
 \begin{equation}\label{blow-up-4-3}
\frac{d}{dt}m_{i}\leq-\frac1{2}m_{i}^2+\frac{\sqrt{3}}{3}|\mu_0|\mu_1-\frac1{2}\mu_1^2+2|\kappa|
(m_2-m_1), \quad i=1,\, 2.
\end{equation}
Summing up the above two inequalities gives
\begin{equation*}\label{blow-up-4-4}
\frac{d}{dt}(m_{1}+m_2)\leq
-\frac1{2}(m_{1}^2+m_2^2)+\mu_1\left(\frac{2\sqrt{3}}{3}|\mu_0|-\mu_1\right)+4|\kappa|
(m_2+m_1)-8|\kappa| m_1.
\end{equation*}
If $\frac{2\sqrt{3}}{3}|\mu_0|< \mu_1$, one has
\begin{equation}\label{blow-up-4-5}
\frac{d}{dt}(m_{1}+m_2)\leq -\frac1{2}(m_{1}^2+m_2^2)+4|\kappa|
(m_2+m_1)-8|\kappa| m_1-2C_1^2.
\end{equation}
Since $(m_1 + m_2)(0)  < -8|\kappa|$, there is $\delta_0\in(0,
\frac{1}{2}]$ such that
$
(m_1 + m_2)(0)  \leq -\alpha
$
with $\alpha=8|\kappa|+\delta_0>8|\kappa|$.

We first claim that there holds
\begin{equation*}(m_1 + m_2)(t)  \leq -\alpha \quad \mbox{for} \quad \forall \quad t
\in (0, T).
\end{equation*}  Indeed,
note that $\bar{m}(t) := (m_1 + m_2)(t) + \alpha$ is continuous on
$[0, T)$. If the above inequality does not hold, we can find a $t_0
\in (0, T)$ such that $\bar{m}(t_0) > 0$. Denote
\begin{equation*}
t_1=\max\{t<t_0|\quad \bar{m}(t_0)=0\}.
\end{equation*}
Then
\begin{equation}\label{blow-up-4-6}
\bar{m}(t_1)=0 \quad \mbox{and} \quad \frac{d}{dt}\bar{m}(t_1) \geq
0.
\end{equation}
While thanks to
\begin{equation*}
m_1(t_1)\leq \frac{1}{2}\bar{m}(t_1)-\frac{\alpha}{2}
=-\frac{\alpha}{2},
\end{equation*}
we get from \eqref{blow-up-4-5} that
\begin{equation*}
\begin{split}
 \frac{d}{dt}\bar{m}(t_1)=\frac{d}{dt}(m_1+m_2)(t_1) &<
-\frac1{2}(m_{1}^2+m_2^2)(t_1)+4|\kappa| (m_2+m_1)(t_1)-8|\kappa|
m_1(t_1)\\
&\leq -\frac1{2}m_{1}^2(t_1)-4|\kappa|\alpha-8|\kappa| m_1(t_1)\\
 &=
-\frac1{2}(m_{1}(t_1)+8|\kappa|)^2-4|\kappa|(\alpha-8|\kappa|)\leq
0.
\end{split}
\end{equation*}
This yields a contradiction with \eqref{blow-up-4-6}, and this completes  the
proof of the claim.

Putting the obtained estimate $m_1(t)\leq \frac{m_1(t) + m_2(t)}{2}
\leq -\frac{\alpha}{2}<-4|\kappa| $ back into \eqref{blow-up-4-3}
with $i=1$, we find
 \begin{equation}\label{blow-up-4-7}
 \begin{split}
\frac{d}{dt}(m_{1}(t)+4|\kappa|)&=\frac{d}{dt}m_{1}(t)\leq-\frac1{2}m_{1}^2-C_1^2+2|\kappa|
(m_2+m_1)-4|\kappa| m_1\\
&\leq-\frac1{2}m_{1}^2-C_1^2-2|\kappa| \alpha-4|\kappa| m_1\\
&\leq
-\frac1{2}(m_{1}+4|\kappa|)^2-2|\kappa|(\alpha-4|\kappa|)-C_1^2\\
&< -\frac{1}{2}(m_{1}+4|\kappa|)^2 \quad \mbox{for a.e.} \quad  t
\in (0, T),
 \end{split}
\end{equation}
which implies $m_{1}(t) +4|\kappa|<0$ on $(0, T)$. From this and the
fact that $m_{1}(t) +4|\kappa|$ is locally Lipshitz on $(0, T)$, we
see that $\frac{1}{m_{1}(t) +4|\kappa|}$ is also Lipshitz on $(0,
T)$. Being locally Lipshitz, the $\frac{1}{m_{1}(t)+4|\kappa|}$ is
absolutely continuous on $(0, T)$, it is then inferred from
\eqref{blow-up-4-7} that
 \begin{equation*}\label{blow-up-4-8}
 \begin{split}
\frac{d}{dt}\left(\frac{1}{m_{1}(t) +4|\kappa|}\right)\geq
\frac{1}{2} \quad \mbox{for a.e.} \quad  t \in (0, T).
 \end{split}
\end{equation*}
Therefore, we get
 \begin{equation*}\label{blow-up-4-9}
 \begin{split}
m_{1}(t)\leq
\frac{2\left(m_{1}(0)+4|\kappa|\right)}{2+\left(m_{1}(0)+4|\kappa|\right)
t}-4|\kappa| \quad \mbox{for a.e.} \quad  t \in (0, T),
 \end{split}
\end{equation*}
which implies that the life-span $T \leq  \frac{-2}{
m_{1}(0)+4|\kappa|}$.

On the other hand, if $\frac{2\sqrt{3}}{3}|\mu_0|\geq \mu_1$, we
find from \eqref{blow-up-4-3} that
 \begin{equation}\label{blow-up-4-10}
\frac{d}{dt}m_{i}\leq-\frac{1}{2}m_{i}^2+2|\kappa| (m_2-m_1)+C_1^2,
\quad i=1,\, 2.
\end{equation}
Summing up the above two inequalities gives
\begin{equation}\label{blow-up-4-11}
\frac{d}{dt}(m_{1}+m_2)\leq -\frac1{2}(m_{1}^2+m_2^2)+4|\kappa|
(m_2+m_1)-8|\kappa| m_1+2C_1^2.
\end{equation}
Since $(m_1 + m_2)(0)  < -8|\kappa|-2\sqrt{2}C_1$, then there is
$\delta_0\in(0, \frac{1}{2}]$ such that $ (m_1 + m_2)(0)  \leq
-\alpha -2\sqrt{2}(1+\delta_0)C_1$ with
$\alpha=8|\kappa|+\delta_0>8|\kappa|$.

Again we first claim that there holds for all $t \in (0, T)$
\begin{equation*}\label{blow-up-4-12}(m_1 + m_2)(t)  \leq
-\alpha -2\sqrt{2}(1+\delta_0)C_1.
\end{equation*} Indeed, similar to the argument above,
note that $\bar{m}(t) := (m_1 + m_2)(t)
+\alpha+2\sqrt{2}(1+\delta_0)C_1$ is continuous on $[0, T)$. If the
above inequality does not hold, we can find a $t_0 \in (0, T)$ such
that $\bar{m}(t_0) \geq 0$. Denote
\begin{equation*}
t_1=\max\{t<t_0|\quad \bar{m}(t_0)=0\}.
\end{equation*}
Then
\begin{equation}\label{blow-up-4-13}
\bar{m}(t_1)=0 \quad \mbox{and} \quad \frac{d}{dt}\bar{m}(t_1) \geq
0.
\end{equation}
While thanks to
\begin{equation*}
m_1(t_1)\leq
\frac{1}{2}\bar{m}(t_1)-\frac{\alpha}{2}-\sqrt{2}(1+\delta_0)C_1=-\frac{\alpha}{2}-\sqrt{2}(1+\delta_0)C_1
\end{equation*}
and
\begin{equation*}
m_2(t_1)=\bar{m}(t_1)-\alpha-2\sqrt{2}(1+\delta_0)C_1-m_1(t_1)=-\alpha-2\sqrt{2}(1+\delta_0)C_1-m_1(t_1),
\end{equation*}
we get from \eqref{blow-up-4-11} that
\begin{equation*}
\begin{split}
& \frac{d}{dt}\bar{m}(t_1)\leq
-\frac1{2}(m_{1}^2+m_2^2)(t_1)+4|\kappa| (m_2+m_1)(t_1)-8|\kappa|
m_1(t_1)+2C_1^2\\
&=
-\frac1{2}m_{1}^2(t_1)-\frac1{2}\left(m_{1}(t_1)+\alpha+2\sqrt{2}(1+\delta_0)C_1\right)^2\\
&\qquad\qquad \qquad-4|\kappa|
\left(\alpha+2\sqrt{2}(1+\delta_0)C_1\right) -8|\kappa|
m_1(t_1)+2C_1^2\\
&=
-\frac{1}{4}\left(2m_{1}(t_1)+\alpha+2\sqrt{2}(1+\delta_0)C_1+8|\kappa|\right)^2
+\frac{1}{4}\left(\alpha+2\sqrt{2}(1+\delta_0)C_1+8|\kappa|\right)^2\\
&\qquad\qquad \qquad
-\frac1{2}\left(\alpha+2\sqrt{2}(1+\delta_0)C_1\right)^2+2C_1^2
-4|\kappa| \left(\alpha+2\sqrt{2}(1+\delta_0)C_1\right),
\end{split}
\end{equation*}
which together with the fact $\alpha>8|\kappa|$ implies
\begin{equation*}
\begin{split}
 \frac{d}{dt}\bar{m}(t_1)
&\leq
\frac{1}{4}\left(\alpha+2\sqrt{2}(1+\delta_0)C_1+8|\kappa|\right)^2
-\frac1{2}\left(\alpha+2\sqrt{2}(1+\delta_0)C_1\right)^2\\
& \qquad +2C_1^2 -4|\kappa|
\left(\alpha+2\sqrt{2}(1+\delta_0)C_1\right)\\
&=-\frac1{4}\left(\alpha+2\sqrt{2}(1+\delta_0)C_1\right)^2+2C_1^2+16|\kappa|^2<0.
\end{split}
\end{equation*}
This yields a contradiction with \eqref{blow-up-4-13}, and
the proof of the claim is complete.

Therefore, $m_1(t)\leq \frac{m_1(t) + m_2(t)}{2} \leq - \frac{\alpha
}{2}-\sqrt{2}(1+\delta_0)C_1<-4|\kappa|-\sqrt{2}(1+\delta_0)C_1$
back into \eqref{blow-up-4-10} with $i=1$, we find for all $ t \in
(0, T)$
 \begin{equation*}\label{blow-up-4-14}
 \begin{split}
&\frac{d}{dt}\left(m_{1}(t)+4|\kappa|\right)=\frac{d}{dt}m_{1}(t)\leq-\frac1{2}m_{1}^2+C_1^2+2|\kappa|
(m_2+m_1)-4|\kappa| m_1\\
&\leq-\frac1{2}m_{1}^2+C_1^2-2|\kappa| \left(\alpha+2\sqrt{2}(1+\delta_0)C_1\right)-4|\kappa| m_1\\
&= -\frac1{2}(m_{1}+4|\kappa|)^2+C_1^2-2|\kappa|
\left(\alpha-4|\kappa|+2\sqrt{2}(1+\delta_0)C_1\right) \\
&\leq -\frac{\delta_0}{2(1+\delta_0)}(m_{1}+4|\kappa|)^2 ,
 \end{split}
\end{equation*}
which implies $m_{1}(t) < -4|\kappa|-\sqrt{2}(1+\delta_0)C_1$ on
$(0, T)$. From this and the fact that $m_1(t)$ is locally Lipshitz
on $(0, T)$, we see that $\frac{1}{m_{1}(t)+4|\kappa|}$ is also
Lipshitz on $(0, T)$. Being locally Lipshitz, the
$\frac{1}{m_{1}(t)+4|\kappa|}$ is absolutely continuous on $(0, T)$,
it is then inferred from \eqref{blow-up-4-7} that
 \begin{equation*}
 \begin{split}
\frac{d}{dt}\left(\frac{1}{m_{1}(t)+4|\kappa|}\right)\geq
\frac{\delta_0}{2(1+\delta_0)} \quad \mbox{for \, a.e.} \quad  t \in
(0, T).
 \end{split}
\end{equation*}
Therefore, we again get
 \begin{equation*}
 \begin{split}
m_{1}(t)\leq
\frac{2(1+\delta_0)(m_{1}(0)+4|\kappa|)}{2(1+\delta_0)+\delta_0\,(m_{1}(0)+4|\kappa|)
t}-4|\kappa| \quad \mbox{for \, a.e.} \quad  t \in (0, T),
 \end{split}
\end{equation*}
which implies that the life-span $T \leq - \frac{2(1+\delta_0)}{
\delta_0 \, (m_{1}(0)+4|\kappa|)}$. This completes the proof of
Theorem \ref{thm-blow-4}.
\end{proof}

\begin{rmk}\label{rmk-blow-differ-1}
Theorem \ref{thm-blow-4} does not overlap with Theorems \ref{t4.3},
or Theorem \ref{t4.5},  which may be easily
verified when we consider the two cases, $\mu_1 \gg |\kappa|
\thicksim |\mu_0|$ and $\mu_1 \ll |\kappa| \thicksim |\mu_0|$
respectively.
\end{rmk}

Using the conserved quantities $H_2$, we can derive the following
wave-breaking result.
\begin{thm}\label{t4.6}
Let $u_0\in H^s(\mathbb{S}),  s > 3/2$ and $T>0$ be the maximal time
of existence of the corresponding solution $u(t,x)$ to \eqref{e1.4}
with the initial data $u_0$. If
\begin{equation}\label{assumption-1-1}
(\mu_0+\kappa)H_2 <\frac{1}{8}\mu_1^4
+\frac{1}{2}\mu_0(\mu_0+\kappa)(2\mu_0^2+\mu_1^2), \quad
\mu_0(\mu_0+\kappa)\geq 0, \quad \mbox{or}
\end{equation}
\begin{equation}\label{assumption-1-2}
(\mu_0+\kappa)H_2<\frac{1}{8}\mu_1^4
+\frac{1}{2}\mu_0(\mu_0+\kappa)\left(2\mu_0^2+(1+\frac{1}{2\pi^2})\mu_1^2\right),
\quad  \mu_0(\mu_0+\kappa) < 0,
\end{equation}
where $\mu_0$, $\mu_1$ are defined in \eqref{2.1} and \eqref{2.2},
then the corresponding solution $u(t,x)$ to \eqref{e1.4} blows up in
finite time $T$ with
\begin{equation*}
0<T\le6+\dfrac{1+\left|\int_{\mathbb{S}}u_{0x}^3(x)\;d x\right|}
{\frac{3}{2}\mu_1^4+6\mu_0(\mu_0+\kappa)(\mu_1^2+2\mu_0^2)-12(\mu_0+\kappa)H_2},
\quad \mbox{if}\quad \mu_0(\mu_0+\kappa)\geq 0
\end{equation*} or
\begin{equation*}0<T\le6+\dfrac{1+\left|\int_{\mathbb{S}}u_{0x}^3(x)\;d x\right|}
{\frac{3}{2}\mu_1^4+\mu_0(\mu_0+\kappa)\left(\left(6+\frac{3}{\pi^2}\right)\mu_1^2+12\mu_0^2\right)-12(\mu_0+\kappa)H_2},
\quad \mbox{if}\quad \mu_0(\mu_0+\kappa)< 0
\end{equation*} such that
\begin{equation*}\liminf\limits_{t\uparrow T}\left(\inf\limits_{x\in\mathbb{S}}u_x(t,x)\right)
=-\infty.
\end{equation*}
\end{thm}
\begin{proof}
Again it suffices to consider the case $s=3$. Recall that
\begin{equation*}H_2=\int_{\Bbb S}\left (\mu_0u^2+\kappa u^2+\frac1{2}uu_x^2 \right )d x
\end{equation*}
is independent of time $ t$. In view of  (\ref{e3.1}), we obtain
\begin{equation}\label{blow-up-2-0}
\begin{split}
 \dfrac{d}{d t}\int_{\mathbb{S}}u_x^3\;d x&=-\frac1{2}\int_{\mathbb{S}}u_x^4\;d x
 -\frac{3}{2}\mu_1^4
 +6(\mu_0+\kappa)\int_{\mathbb{S}}uu_x^2\;d x\\
 &\qquad \qquad \qquad \qquad\qquad -6(\mu_0+\kappa)\mu_0\int_{\mathbb{S}}u_x^2\;d x\\
 &=-\frac1{2}\int_{\mathbb{S}}u_x^4\;d x-\frac{3}{2}\mu_1^4+12(\mu_0+\kappa)H_2
-6\mu_0(\mu_0+\kappa)\mu_1^2\\
 &\qquad \qquad \qquad \qquad\qquad-12\mu_0(\mu_0+\kappa)\int_{\mathbb{S}}u^2\;d x.
\end{split}
\end{equation}
If $\mu_0(\mu_0+\kappa)\geq 0$, it then follows from H\"{o}lder's
inequality that
$$\mu_0(\mu_0+\kappa)\int_{\mathbb{S}}u^2\;d x \geq
\mu_0(\mu_0+\kappa)(\int_{\mathbb{S}}u\;d
x)^{2}=\mu_0^3(\mu_0+\kappa).$$ Hence, we have
\begin{equation}\label{blow-up-2-1}
\begin{split}
&\frac{3}{2}\mu_1^4-12(\mu_0+\kappa)H_2
+6\mu_0(\mu_0+\kappa)\mu_1^2+12\mu_0(\mu_0+\kappa)\int_{\mathbb{S}}u^2\;d
x\\
&\geq
\frac{3}{2}\mu_1^4+6\mu_0(\mu_0+\kappa)(\mu_1^2+2\mu_0^2)-12(\mu_0+\kappa)H_2=:C_0
\end{split}
\end{equation}
Thanks to the assumption \eqref{assumption-1-1}, we get $C_0>0$.

On the other hand, if $\mu_0(\mu_0+\kappa) < 0$, we get from
 \eqref{e2.3} that
\begin{equation*}
\begin{split}\mu_0(\mu_0+\kappa)\int_{\mathbb{S}}u^2\;d x
\geq
\mu_0(\mu_0+\kappa)\left(\frac{1}{4\pi^2}\mu_1^2+\mu_0^2\right).
\end{split}
\end{equation*}
It then follows that
\begin{equation}\label{blow-up-2-2}
\begin{split}
&\frac{3}{2}\mu_1^4-12(\mu_0+\kappa)H_2
+6\mu_0(\mu_0+\kappa)\mu_1^2+12\mu_0(\mu_0+\kappa)\int_{\mathbb{S}}u^2\;d
x\\
&\geq
\frac{3}{2}\mu_1^4+\mu_0(\mu_0+\kappa)\left((6+\frac{3}{\pi^2})\mu_1^2+12\mu_0^2\right)
-12(\mu_0+\kappa)H_2=:C_0
\end{split}
\end{equation}
Thanks to the assumption \eqref{assumption-1-1}, we also get
$C_0>0$.

In view of  \eqref{blow-up-2-0}-\eqref{blow-up-2-2},
together with H\"{o}lder's inequality applied, we deduce that
 $$\frac{d}{d t}\int_{\mathbb{S}}u_x^3\;d x\le -\frac{1}{2}\int_{\mathbb{S}}u_x^4\;d x-C_0\le
 -\frac{1}{2}\left (\int_{\mathbb{S}}u_x^3\;d x\right )^{\frac{4}{3}}-C_0.$$
Define  $V(t)=\int_{\mathbb{S}}u_x^3(t, x)\ d x$ with $ t \in [0,
T).$ It is clear that
 $$\dfrac{d}{d t}V(t)\le-\frac{1}{2}(V(t))^{\frac{4}{3}}-C_0\le-C_0<0,\quad t\in[0,T).$$
 Let $t_1=(1+|V(0)|)/C_0$. Then following the proof of Theorem
\ref{t4.3},  we have
 $$T\le t_1+6<+\infty.$$
This implies the desired result as in Theorem \ref{t4.6}.
\end{proof}

\setcounter{equation}{0}

\section{Existence of global strong solution}

In this section, attention is now turned to specifying conditions under which the local
strong solution to the initial-value problem \eqref{e1.1} can be extended to a global one.
\begin{thm}\label{thm-gss-1}
If the initial potential $m_0 \in H^1(\mathbb{S})$ satisfies that
$m_0+\kappa $ does not change the sign, then the solution $ u(t) $
to the initial-value problem \eqref{e1.1} exists permanently in
time.
\end{thm}
\begin{proof}
Let $T$ be the maximal time of existence of the solution $u$ to
\eqref{e1.4} with the initial data $u_0$, guaranteed by Proposition
\ref{p2.1}.

Assume $m_0 +\kappa \geq 0$. We prove that the solution $u(t, x)$
exists globally in time. Indeed, thanks to Lemma \ref{lem-diff-1}
and Remark \ref{rmk-diff-1}, we find $m(t)+\kappa\geq 0$ on $[0, T)
\times \mathbb{S}$. Given $t \in [0, T)$, by the periodicity in the
$x$-variable, there is a $\xi(t) \in (0, 1)$ such that $u_x(t,
\xi(t))=0$. Therefore, for $x \in [\xi(t), \xi(t)+1]$ we have
\begin{equation*}
\begin{split}
-u_x(t, x)=-\int_{\xi(t)}^{x}\partial_x^2\,u(t, x)\,
dx=\int_{\xi(t)}^{x}\left(m(t, x)+\kappa\right)\,
dx-\int_{\xi(t)}^{x}[\mu(u)+\kappa]\,dx,
\end{split}
\end{equation*}
which leads to
\begin{equation}\label{lower-bound-1}
\begin{split}
&-u_x(t, x)\leq \int_{\xi(t)}^{\xi(t)+1}\left(m(t,
x)+\kappa\right)\,
dx-(\mu_0+\kappa)(x-\xi(t))\\
&= \int_{\mathbb{S}}\left(m_0+\kappa\right)\,
dx-(\mu_0+\kappa)(x-\xi(t))=(\mu_0+\kappa)(1-x+\xi(t)) \leq
|\mu_0+\kappa|.
\end{split}
\end{equation}

On the other hand, if $m_0 +\kappa \leq 0$, then $m(t)+\kappa\leq 0$
on $[0, T) \times \mathbb{S}$. Using the same notation as above, we
find that
\begin{equation}\label{lower-bound-2}
\begin{split}
&-u_x(t, x)=-\int_{\xi(t)}^{x}\partial_x^2\,u(t, x)\,
dx=\int_{\xi(t)}^{x}[m(t, x)+\kappa]\,
dx-\int_{\xi(t)}^{x}[\mu(u)+\kappa]\,dx\\
&\leq -(\mu_0+\kappa)(x-\xi(t))\leq |\mu_0+\kappa|.
\end{split}
\end{equation}
From \eqref{lower-bound-1} and \eqref{lower-bound-2}, we deduce that
$u$ exists permanently as a consequence of Proposition \ref{t4.2}.
\end{proof}

\begin{thm}\label{thm-gss-2}
If the initial profile $u_0 \in H^3(\mathbb{S})$ is such that
\begin{equation}\label{global-assump-1}
\|\partial_x^3 u_0\|_{L^2} \leq 2\sqrt{3} |\mu_0+\kappa|,
\end{equation} then the initial-value problem \eqref{e1.4} admits global solutions in time.
\end{thm}

\begin{proof}
 Let $T$ be the maximal time of existence of the solution $u$ to
 \eqref{e1.4} with the initial data $u_0$, given by Proposition \ref{p2.1}.

By Lemma \ref{rmk-2.1}, we get
\begin{equation*}\max (\partial_x^2 u_0)^2 \leq \frac{1}{12}
\int_{\mathbb{S}}(\partial_x^3 u_0)^2\, dx,
\end{equation*}
which gives rise to
\begin{equation}\label{global-0}\|\partial_x^2 u_0\|_{L^{\infty}} \leq \frac{\sqrt{3}}{6}
\|\partial_x^3 u_0\|_{L^2}.
\end{equation}

If $\mu_0+\kappa\geq 0$, it then is inferred from \eqref{global-0} and the
assumption \eqref{global-assump-1} that
\begin{equation*}\label{global-1}m_0+\kappa=\mu_0+\kappa-\partial_x^2 u_0
 \geq \mu_0+\kappa-\frac{\sqrt{3}}{6}
\|\partial_x^3 u_0\|_{L^2} \geq 0
\end{equation*}
Similarly, if $\mu_0+\kappa\leq 0$, one obtains from
\eqref{global-0} and \eqref{global-assump-1} that
\begin{equation*}\label{global-2}m_0+\kappa=\mu_0+\kappa-\partial_x^2 u_0
 \leq \mu_0+\kappa+\frac{\sqrt{3}}{6}
\|\partial_x^3 u_0\|_{L^2} \leq 0.
\end{equation*}
Therefore, in view of Theorem \ref{thm-gss-1}, the proof of this
theorem is complete.
\end{proof}

\setcounter{equation}{0}
\section{Existence of global weak solution}

In this section, we establish the existence of an admissible global
weak solution to \eqref{e1.4}, which may be stated as follows.
\begin{thm}\label{thm-main-1}
Assume that $u_0 \in H^{1}(\mathbb{S})$. Then the initial-value problem
\eqref{e1.1} has an admissible global weak solution, $u=u(t, x)$, in
the sense of Definition \ref{def-gws-1-1}. Furthermore, this weak
solution $u(t, x)$ satisfies the following properties.
\vskip 0.1cm

(i) One-sided supernorm estimate: There exists a positive constant
$C=C(u_0)$ such that the following one-sided $L^{\infty}$ norm
estimate on the first-order spatial derivative holds in the sense of
distribution:
\begin{equation}\label{entropy-time-1}
\partial_x u(t, x) \leq \frac{1}{t}+C, \quad \forall \quad t>0, \, x
\in \mathbb{S}.
\end{equation}

(ii) Space-time higher integrability estimate.

\begin{equation*}
\partial_x u \in L^{p}_{loc}(\mathbb{R}^{+}\times
\mathbb{S}), \quad \forall \, 1 \leq p<3,
\end{equation*}
i.e., for any $0<T<+\infty$, there exists a positive constant $C_1
=C_1(T, p)$ such that
\begin{equation}\label{high-integrability-1-1}
\int_0^{T}\int_{\mathbb{S}}|\partial_x u_{\varepsilon}(t, x)|^{p} \,
dx\, dt\leq C_1, \quad \forall \, 1 \leq p<3.
\end{equation}
\end{thm}

The proof of this theorem is motivated by the one of Theorem 1.2 in
\cite{xz}.  This method, as far as we know,  was first used by Zhang
and Zheng to study an admissible global solution to a variational
wave equation in \cite{zhangz-4}.

\subsection{Viscous Approximate Solutions}

We obtain the existence of a global weak solution to the initial-value problem
\eqref{e1.4} by proving compactness of a sequence of smooth
functions $\{u_{\varepsilon}\}_{\varepsilon>0}$ solving the
following viscous problems with the initial data $u_{\varepsilon
0}(x)=\phi_{\varepsilon} \ast u_{0}$,
\begin{equation}\label{appr-equation-1}
\begin{cases}
\partial_t u_{\varepsilon}+u_{\varepsilon} \partial_x u_{\varepsilon}+\partial_x P_{\varepsilon}-\varepsilon \partial_x^2
u_{\varepsilon}=0,
 \qquad t > 0, \quad x \in \mathbb{R}, \\
(\mu-\partial_x^2) P_{\varepsilon}=2\mu(u_{\varepsilon})
u_{\varepsilon}+\frac{1}{2}(\partial_x u_{\varepsilon})^2+2\kappa\, u_{\varepsilon},
\quad t > 0, \quad x \in \mathbb{R},\\
u_{\varepsilon}(t,x+1)=u_{\varepsilon}(t,x),\quad\quad\quad t \ge 0,
\quad x\in\mathbb{R}, \\
u_{\varepsilon}(0,x)=u_{\varepsilon 0}(x),  \qquad x \in \mathbb{R},
 \end{cases}
\end{equation}
or equivalently,
\begin{equation}\label{appr-equation-1-ab}
\begin{cases}
\partial_tm_{\varepsilon} -\varepsilon
\partial_x^2
m_{\varepsilon}+2\kappa\,\partial_x
u_{\varepsilon}+u_{\varepsilon}\partial_x m_{\varepsilon}+2
m_{\varepsilon}\partial_x u_{\varepsilon}=0,\quad t > 0, \quad x \in \mathbb{R}, \\
m_{\varepsilon}= (\mu-\partial_x^2)u_{\varepsilon}, \quad\quad\quad t \ge 0,\quad x\in\mathbb{R},\\
u_{\varepsilon}(t,x+1)=u_{\varepsilon}(t,x),\quad\quad\quad t \ge 0,
\quad x\in\mathbb{R},\\
u_{\varepsilon}(0,x)=u_{\varepsilon 0}(x),  \qquad x \in \mathbb{R},
 \end{cases}
\end{equation}
where the truncating family
$\{\phi_{\varepsilon}(x)\}_{\varepsilon>0} $ satisfies
\begin{equation}\label{truncating-1}
\phi_{\varepsilon}(x) =\varepsilon^{-1}\phi(x/{\varepsilon}) \quad
\mbox{with} \quad \varepsilon>0, \quad \phi \in
C^{\infty}_{c}(\mathbb{R}), \quad \phi \geq 0, \quad
\|\phi\|_{L^1}=1.\end{equation}

The existence, uniqueness, and basic energy estimate on this
approximate solution sequence are given in the following
proposition.
\begin{prop}\label{thm-regu-1}
Let $\varepsilon>0$ and $u_{0\varepsilon}\in H^k(\mathbb{S})$ for
some $k \geq 1$. Then there exists a unique solution
$u_{\varepsilon} \in C(\mathbb{R}^+; H^{k}(\mathbb{S}))$ to the
initial-value problem \eqref{appr-equation-1}. Furthermore, the following
energy identities hold for all $t \geq 0$.
\begin{equation}\label{identity-2}
\begin{split}
\mu(u_\varepsilon(t))=\mu(u_{0\varepsilon}) \quad \mbox{and} \quad
\int_{\mathbb{S}}(\partial_x u_{\varepsilon})^2 (t, x) d
x+2\varepsilon\int_{\Bbb S}(\partial_x^2 u_\varepsilon)^2(t, x)\, d
x=\int_{\mathbb{S}}(\partial_x u_{0\varepsilon})^2 d x.
\end{split}
\end{equation}
\end{prop}
\begin{rmk}\label{rmk-regu-1} Thanks to \eqref{truncating-1},
together with Young's inequality applied, we deduce that
\begin{equation*}\label{regu-1-1}
\begin{split}
  \mu(u_{0\varepsilon})&=\int_{\mathbb{S}}\int_{\mathbb{R}}
  \frac{1}{\varepsilon} \phi(\frac{y}{\varepsilon}) u_0(x-y) \, dy dx
  =\int_{\mathbb{R}} \frac{1}{\varepsilon} \phi(\frac{y}{\varepsilon})(\int_{\mathbb{S}}
  u_0(x-y) \, dx) dy\\
  &=\mu(u_0) \int_{\mathbb{R}}\frac{1}{\varepsilon} \phi(\frac{y}{\varepsilon})\, dy=\mu(u_0)=\mu_0
\end{split}
\end{equation*}
and
\begin{equation}\label{regu-1-2}
\begin{split}
 \int_{\mathbb{S}}(\partial_x u_{0\varepsilon})^2 d x
 =\|\phi_{\varepsilon} \ast \partial_x u_{0}\|_{L^2}^2
 \leq \|\phi_{\varepsilon}\|_{L^1} \|\partial_x u_{0}\|_{L^2}^2=\|\partial_x
 u_{0}\|_{L^2}^2=\mu_1^2.
\end{split}
\end{equation}
\end{rmk} The strategy of the proof of Proposition \ref{thm-regu-1} is rather routine.
For the sake of simplicity, we will only sketch the necessary
estimates. While for the convenience of presentation, we will omit
the subscript $\varepsilon$ in $u_{\varepsilon}$ in the following
proof.
\begin{proof}[Proof of Proposition \ref{thm-regu-1}]
First, following the standard argument for a nonlinear parabolic
equation, one can obtain the local well-posedness result that for
$u_{0\varepsilon} \in H^k(\mathbb{S})$, there exists a positive
constant $T_0$ such that \eqref{appr-equation-1} has a unique
solution
\begin{equation*}u=u(t, x) \in C([0, T_0],
H^{k}(\mathbb{S})) \cap L^2([0, T_0], H^{k+1}(\mathbb{S})).
\end{equation*}
We denote the life span of the solution $u(t, x)$  by $T$. Then,
\eqref{identity-2} holds for all $0 \leq t < T$.

Next we claim that if the life span $T < +\infty$, i.e., $u\in C([0,
T), H^{k}(\mathbb{S}))$, and
\begin{equation}\label{blow-up-app-1-ab}\lim_{t \rightarrow T}\|u(t, \cdot)\|_{H^k(\mathbb{S})}
=+\infty, \quad T<+\infty,\end{equation}
then
\begin{equation*}\label{blow-up-app-1}
\lim_{t \rightarrow
T}\int_0^t\|\partial_xu(\tau,\cdot)\|_{L^{\infty}(\mathbb{S})}\,
d\tau=+\infty, \quad T<+\infty.
\end{equation*} Indeed, assume that the maximal existence time $T<+\infty$.
It then follows from the equation in \eqref{appr-equation-1}, together
with \eqref{identity-2}, that for $t <T$
\begin{equation}\label{energy-1-1}
\begin{split}
&\frac{1}{2}\frac{d}{dt}\| u\|_{H^{k}(\mathbb{S})}^2+\varepsilon
\|\partial_{x}
u\|_{H^{k}(\mathbb{S})}^2\\
&=\sum_{\alpha=0}^{k}\int_{\mathbb{S}}\left(\frac{1}{2}\partial_x
u(\partial_x^{\alpha}u)^2+(u
\partial_x^{1+\alpha}u-\partial_x^{\alpha}(u\partial_x u))\partial_x^{\alpha}u
-\partial_x^{1+\alpha}P\partial_x^{\alpha}u\right) \, dx.
\end{split}
\end{equation}
Note that
\begin{equation}\label{inequality-1-1}
\begin{split}
&\sum_{\alpha=0}^{k}\int_{\mathbb{S}}\left(\frac{1}{2}\partial_x
u(\partial_x^{\alpha}u)^2\right)\, dx \leq C_{k} \|\partial_x
u(t)\|_{L^{\infty}(\mathbb{S})}\| u(t)\|_{H^{k}(\mathbb{S})}^2
\end{split}
\end{equation}
and \begin{equation}\label{inequality-1-3}
\begin{split}
|\int_{\mathbb{S}} \partial_x^{1+\alpha}P\partial_x^{\alpha}u\,
dx|\leq \|
\partial_x^{\alpha}u(t)\|_{L^{2}(\mathbb{S})}\|
\partial_x^{1+\alpha}P\|_{L^{2}(\mathbb{S})}.
\end{split}
\end{equation}
Due to Remark \ref{rmk-a-opera-1}, we apply a standard elliptic
regularity estimate to \eqref{formula-2.2} to obtain that for $0\leq
\alpha \leq k$
\begin{equation}\label{energy-1-6}
\begin{split}
\|\partial_x^{1+\alpha}P\|_{L^{2}} &\leq C\left(|\mu(u)|^2+\|
u\|_{H^{1}}^2+\kappa
\|u\|_{L^{2}(\mathbb{S})}+\|\partial_x^{\alpha-1} ((\partial_x
u)^2)\|_{L^{2}}^2\right)\\
 &\leq C\left(|\mu(u)|^2+\|
u\|_{H^{1}}^2+\kappa \|u\|_{L^{2}(\mathbb{S})}+\|\partial_x
u\|_{L^{\infty}}\| u\|_{H^{\alpha}}\right) .
\end{split}
\end{equation}
Applying the Kato-Ponce commutator estimate \cite{kat} yields
\begin{equation}\label{inequality-1-2}
\begin{split}
\|u
\partial_x^{1+\alpha}u-\partial_x^{\alpha}(u\partial_x u)\|_{L^{2}(\mathbb{S})}
\leq C_{k}\|\partial_x u(t)\|_{L^{\infty}(\mathbb{S})}\|
\partial_x^{\alpha}u(t)\|_{L^{2}(\mathbb{S})},
\end{split}
\end{equation}
which, together with \eqref{inequality-1-1}, \eqref{inequality-1-3},
\eqref{energy-1-6} and \eqref{identity-2} applied to
\eqref{energy-1-1}, leads to
\begin{equation}\label{energy-1-7}
\begin{split}
&\frac{d}{dt}\| u\|_{H^{k}(\mathbb{S})}^2+2\varepsilon
\|\partial_{x} u\|_{H^{k}(\mathbb{S})}^2\leq C_{k} (\|\partial_x
u\|_{L^{\infty}}+1)\| u\|_{H^{k}(\mathbb{S})}^2.
\end{split}
\end{equation}
Hence, if $\lim_{t \rightarrow
T}\int_0^t\|\partial_xu(\tau,\cdot)\|_{L^{\infty}(\mathbb{S})}\,
d\tau<+\infty$, then applying Gronwall's inequality to
\eqref{energy-1-7}, we get $\lim_{t \rightarrow T}\|u(t,
\cdot)\|_{H^k(\mathbb{S})}<+\infty$, which contradicts
\eqref{blow-up-app-1-ab}. This completes the proof of the claim.

On the other hand, thanks to Lemma \ref{l2.1}, we get
\begin{equation*}\max\limits_{x\in \mathbb{S}}\;\left(\partial_x u(t,x)\right)^2
\leq\frac{1}{12}\int_{\mathbb{S}}(\partial_x^2u)^2(t,x)  dx.
\end{equation*}
 From this, together with \eqref{identity-2},
\eqref{regu-1-2} and H\"{o}lder's inequality, we obtain that for any
$0 \leq t <T$
\begin{equation*}
\begin{split}
\int_0^t\|\partial_xu(\tau,\cdot)\|_{L^{\infty}(\mathbb{S})}\, d\tau
 \leq \frac{\sqrt{3}}{6} \int_0^t\|\partial_x^2
u(\tau,\cdot)\|_{L^{2}(\mathbb{S})}\, d\tau \leq
\frac{\sqrt{6}}{12\sqrt{\varepsilon}}T^{\frac{1}{2}}\|\partial_x
u_0\|_{L^{2}(\mathbb{S})},
\end{split}
\end{equation*}
which implies that the life-span $T=+\infty$. Furthermore,
\eqref{identity-2} now holds on $[0, +\infty)$. This completes the
proof of Proposition \ref{thm-regu-1}.
\end{proof}

\subsection{Uniform {\it A Priori} Estimates}

Let $u_0 \in H^1(\mathbb{S})$ and $u_{\varepsilon}(t, x)$ be the
unique global-in-time solution to \eqref{appr-equation-1} obtained
in Proposition \ref{t4.2} which satisfies the energy
identity \eqref{identity-2}. To obtain the compactness of this
approximate solution sequence, we need some {\it a priori} estimates
in addition to \eqref{identity-2}. In this subsection, we derive the
uniform one-sided supernorm estimate \eqref{entropy-time-1} and the
space-time higher integrability estimates
\eqref{high-integrability-1-1} on $\partial_x u_{\varepsilon}(t,
x)$, which are essential for our compactness argument.

We start with the uniform one-sided supernorm estimate, which is
similar to Oleinik's entropy condition for the theory of shock waves
\cite{xin}.

\begin{prop}\label{prop-entrop-con}
There holds
\begin{equation}\label{entrop-con-1}
\partial_x u_{\varepsilon}(t, x) \leq \frac{2}{t}+L_0, \quad \forall
\, t>0, \quad x \in \mathbb{S}
\end{equation}
with the constant
\begin{equation*}\label{entrop-const-2}L_0:=\sqrt{2(\mu_0+\kappa)^2+\frac{7}{6}\mu_1^2} .
\end{equation*}
\end{prop}
\begin{proof}
Set $q_{\varepsilon}=\partial_x u_{\varepsilon}$. Differentiating
the first equation in \eqref{appr-equation-1} with respect to $x$,
we get from Proposition \ref{thm-regu-1} and Remark \ref{rmk-regu-1}
that
\begin{equation}\label{one-derivative-1}
\begin{cases}
\partial_t q_{\varepsilon}+u_{\varepsilon}\partial_x
q_{\varepsilon}
-\varepsilon\partial_x^2q_{\varepsilon}+\frac{1}{2}(q_{\varepsilon})^2
=2(u_{\varepsilon}-\mu_0)(\mu_0+\kappa)-\frac{1}{2}\mu(q_{\varepsilon}^2),\\
q_{\varepsilon}(t, x)|_{t=0}=\partial_x u_{0\varepsilon}.
\end{cases}
\end{equation}
Thanks to \eqref{max-1-1}(up to a slight modification), one has
\begin{equation*}\label{integral-1-1}
\begin{split}
  \|2(u_{\varepsilon}-\mu_0)(\mu_0+\kappa)\|_{L^{\infty}}
  \leq   \frac{\sqrt{3}}{3}|\mu_{0}+\kappa|\mu_1\leq (\mu_0+\kappa)^2+\frac{1}{12}\mu_1^2.
\end{split}
\end{equation*}
While from \eqref{identity-2}, we deduce that
\begin{equation*}\label{integral-1-2}
\begin{split}
  \|\frac{1}{2}\mu(q_{\varepsilon}^2)\|_{L^{\infty}} \leq \frac{1}{2}\mu_1^2.
\end{split}
\end{equation*}
So,
\begin{equation}\label{integral-1-3}
\begin{split}
  \|2(u_{\varepsilon}-\mu_0)(\mu_0+\kappa)-\frac{1}{2}\mu(q_{\varepsilon}^2)\|_{L^{\infty}}
  \leq (\mu_0+\kappa)^2+\frac{7}{12}\mu_1^2
    =\frac{1}{2} L_0^2.
\end{split}
\end{equation}
Define $Q_{\varepsilon}(t)$(for $t>0$) which solves the following
ordinary differential equation
\begin{equation}\label{one-derivative-2}
\begin{cases}
\dfrac{d}{dt} Q_{\varepsilon}+\frac{1}{2}(Q_{\varepsilon})^2
=\frac{1}{2} L_0^2,\\
Q_{\varepsilon}(t=0)=\max\{0, \partial_x u_{0\varepsilon}\}.
\end{cases}
\end{equation}
Then the function $Q_{\varepsilon}(t)$ is a supersolution of the
parabolic initial-value problem \eqref{one-derivative-1}. The
comparison principle for parabolic equations leads to
\begin{equation}\label{supersolution-1}
q_{\varepsilon}(t, x) =\partial_x u_{\varepsilon}(t, x) \leq
Q_{\varepsilon}(t), \quad \forall\, t \geq 0, \quad x \in
\mathbb{S}.
\end{equation}
While a direct computation yields that $L(t):= \frac{2}{t}+L_0$(with
$t>0$) satisfies
\begin{equation*}\label{one-derivative-3}
\dfrac{d}{dt}
L(t)+\frac{1}{2}L(t)^2=\frac{1}{2}L_0^2+\frac{2L_0}{t}>\frac{1}{2}L_0^2,
\quad \forall \quad t>0,
\end{equation*}
which implies that $L(t)$ is a supersolution of
\eqref{one-derivative-2}. Hence, the comparison principle for a
parabolic equation yields $Q_{\varepsilon}(t) \leq L(t)$  for all
$t>0$, which together with \eqref{supersolution-1} admits
\eqref{entrop-con-1}.
\end{proof}

Next, we establish the uniform local space-time higher integrability
estimate \eqref{high-integrability-1-1} motivated by the idea in
\cite{lions, xz, zhangz-3, zhangz-4, zhangz-2}, which is crucial to
studying the structures of the Young measure associated with the
weak convergence sequence $\partial_x u_{\varepsilon}$.

\begin{prop}\label{prop-high-inte}
Let $0 < \alpha< 1$, $T > 0$. Then there exists a positive constant
$C$ depending only on $\|u_0\|_{H^1}$, $T$( but independent of
$\varepsilon$) such that
\begin{equation}\label{high-integrability}
\int_0^{T}\int_{\mathbb{S}}|\partial_x u_{\varepsilon}(t,
x)|^{2+\alpha} \, dx\, dt\leq C,
\end{equation}
where $u_{\varepsilon} = u_{\varepsilon}(t, x)$ is the unique
solution of \eqref{appr-equation-1}.
\end{prop}
First, a direct computation yields that
\begin{lem}\label{lemma-cutoff-1}(\cite{chk})
For every $0<\alpha <1$, the function $\theta(\xi)=
\xi(|\xi|+1)^{\alpha}$ with $\xi \in \mathbb{R}$ satisfies the
following property.
\begin{equation*}
\theta'(\xi)=\left((\alpha+1)|\xi|+1\right)(|\xi|+1)^{\alpha-1},
\end{equation*}
\begin{equation*}
\theta''(\xi)=\alpha(\alpha+1)\mbox{sign}(\xi)(|\xi|+1)^{\alpha-1}
+\alpha(1-\alpha)\mbox{sign}(\xi)(|\xi|+1)^{\alpha-2},
\end{equation*}
\begin{equation*}
\xi\theta(\xi)-\frac{1}{2}\xi^2\theta'(\xi)\geq
\frac{1-\alpha}{2}\xi^2(|\xi|+1)^{\alpha}
\end{equation*}
and
\begin{equation*}
|\theta(\xi)|\leq |\xi|^{\alpha+1}+|\xi|,\quad |\theta'(\xi)|\leq
(\alpha+1)|\xi|+1, \quad |\theta''(\xi)|\leq 2\alpha.
\end{equation*}
\end{lem}
We are now  in a position to prove Proposition \ref{prop-high-inte}.

\begin{proof}[Proof of Proposition \ref{prop-high-inte}]
Multiplying the equation in \eqref{one-derivative-1} by
$\theta'(q_{\varepsilon})$, we get
\begin{equation*}
\begin{split}
\partial_t \theta(q_{\varepsilon})+u_{\varepsilon}\partial_x
\theta(q_{\varepsilon})
&-\varepsilon\theta'(q_{\varepsilon})\partial_x^2q_{\varepsilon}
+\frac{1}{2}\theta'(q_{\varepsilon})(q_{\varepsilon})^2\\
&
=\theta'(q_{\varepsilon})\left(2(u_{\varepsilon}-\mu_0)(\mu_0+\kappa)
-\frac{1}{2}\mu(q_{\varepsilon}^2)\right).
\end{split}
\end{equation*}
Integrating the above equation over $[0, T] \times \mathbb{S}$, we
obtain by integration by parts that
\begin{equation}\label{inte-bdd-0}
\begin{split}
&\int_0^T\int_{\mathbb{S}}[q_{\varepsilon} \theta(q_{\varepsilon}) -
\frac{1}{2}(q_{\varepsilon})^2\theta'(q_{\varepsilon})]\,
dxd\tau\\
&=\int_{\mathbb{S}} \theta(q_{\varepsilon})(T) \,
dx-\int_{\mathbb{S}} \theta(q_{\varepsilon})(0) \,
dx+\varepsilon\int_0^T\int_{\mathbb{S}}(\partial_x
q_{\varepsilon})^2 \theta''(q_{\varepsilon}) \,
dxd\tau\\
&\qquad-\int_0^T\int_{\mathbb{S}}\left(2(u_{\varepsilon}-\mu_0)(\mu_0+\kappa)
-\frac{1}{2}\mu(q_{\varepsilon}^2)\right) \theta'(q_{\varepsilon})
\, dxd\tau.
\end{split}
\end{equation}
It follows from Lemma \ref{lemma-cutoff-1}  that
\begin{equation}\label{inte-bdd-1}
\begin{split}
\int_0^T\int_{\mathbb{S}}\left(q_{\varepsilon}
\theta(q_{\varepsilon}) -
\frac{1}{2}(q_{\varepsilon})^2\theta'(q_{\varepsilon})\right)\,
dxd\tau \geq \frac{1-\alpha}{2}\int_0^T\int_{\mathbb{S}}
|q_{\varepsilon}|^{2+\alpha}\, dxd\tau
\end{split}
\end{equation}
and
\begin{equation}\label{inte-bdd-2}
\begin{split}
|\int_{\mathbb{S}} \theta(q_{\varepsilon})(T) \, dx| &\leq
\int_{\mathbb{S}}(|q_{\varepsilon}(T)|^{1+\alpha}+|q_{\varepsilon}(T)|)
\,dx \\
&\leq \|q_{\varepsilon}(T)\|_{L^2(\mathbb{S})}^{1+\alpha}
+\|q_{\varepsilon}(T)\|_{L^2(\mathbb{S})} \leq
\mu_1^{1+\alpha}+\mu_1.
\end{split}
\end{equation}
Similarly, we have
\begin{equation}\label{inte-bdd-3}
\begin{split}
|\int_{\mathbb{S}} \theta(q_{\varepsilon})(0) \, dx| \leq
\mu_1^{1+\alpha}+\mu_1.
\end{split}
\end{equation}
On the other hand, thanks to \eqref{identity-2} and
\eqref{integral-1-3}, together with Lemma \ref{lemma-cutoff-1}
applied again, we deduce that
\begin{equation}\label{inte-bdd-4}
\begin{split}
\varepsilon\int_0^T\int_{\mathbb{S}}(\partial_x q_{\varepsilon})^2
\theta''(q_{\varepsilon}) \, dxd\tau \leq2 \alpha \varepsilon
\|\partial_x q_{\varepsilon}\|_{L^{2}([0, T] \times \mathbb{S})}^2
\leq \alpha \mu_1^2
\end{split}
\end{equation}and
\begin{equation}\label{inte-bdd-5}
\begin{split}
&\int_0^T\int_{\mathbb{S}}\left(2(u_{\varepsilon}-\mu_0)(\mu_0+\kappa)
-\frac{1}{2}\mu(q_{\varepsilon}^2))
\theta'(q_{\varepsilon}\right) \, dxd\tau \\
&\leq
\|2(u_{\varepsilon}-\mu_0)(\mu_0+\kappa)-\frac{1}{2}\mu(q_{\varepsilon}^2)\|_{L^{\infty}([0,
T] \times \mathbb{S})}\int_0^T\int_{\mathbb{S}}
\left((1+\alpha)|q_{\varepsilon}|+1\right)\, dxd\tau \\
&\leq \frac{L_0^2}{2}T\left(1+(1+\alpha)\mu_1\right).
\end{split}
\end{equation}
Therefore, plunging \eqref{inte-bdd-1}-\eqref{inte-bdd-5} into
\eqref{inte-bdd-0}, we get \eqref{high-integrability}, which completes
the proof of Proposition \ref{prop-high-inte}.
\end{proof}

\subsection{Precompactness}

In this subsection, we drive the theory of Young measures (see Lemma
4.2 in \cite{xz}, also \cite{young}, \cite{zhangz},\cite{zhangz-3},
\cite{zhangz-2}) to obtain the necessary compactness of the viscous
approximate solution $u_{\varepsilon}(t, x)$. We first state a
compactness lemma.
\begin{lem}\label{lemma-compactness-1}(\cite{simon})
Let $X$, $B$, $Y$ be three Banach spaces and satisfy $X
\hookrightarrow \hookrightarrow B \hookrightarrow Y$, $1 \leq p \leq
\infty$, $T>0$. Assume that a set $F$ of functions $f$ is bounded in
$L^p([0, T], X)$ and satisfies that
\begin{equation*}
\begin{split}
\|f(\cdot+h)-f(\cdot)\|_{L^{p}(0, T-h; Y)} \rightarrow 0 \quad
\mbox{as} \quad h \rightarrow 0, \quad \mbox{uniformly for}\quad f
\in F.
\end{split}
\end{equation*}
Then $F$ is relatively compact in $L^p([0, T], B)$ (and in $C([0,
T], B)$ if $p= \infty$).
\end{lem}
With this compactness lemma in hand, we can prove the weak
convergence property in $L^{\infty}(\mathbb{R}^+;
H^1({\mathbb{S}}))$.
\begin{prop}\label{prop-compactness-1}
There exist a subsequence $\{u_{\varepsilon_{j}}(t, x),\,
\mu((\partial_xu_{\varepsilon_j}(t, x))^2)\}$ of the sequence
$\{u_{\varepsilon}(t, x), \, \mu((\partial_xu_{\varepsilon}(t,
x))^2)\}$ and some functions $\{u(t, x), \,\Pi_1(t)\}$, $u \in
L^{\infty}(\mathbb{R}^+; H^1({\mathbb{S}}))$ and $\Pi_1(t) \in
L^{\infty}(\mathbb{R}^+)$, such that
\begin{equation}\label{conti-u}u_{\varepsilon_{j}}\rightarrow u \quad \mbox{as} \quad j \rightarrow +\infty \quad
\mbox{uniformly on each compact subset of} \quad \mathbb{R}^+ \times
\mathbb{S}
\end{equation} and
\begin{equation}\label{strong-p}\mu((\partial_xu_{\varepsilon_j}(t, x))^2)\rightarrow \Pi_1(t)
\quad \mbox{in} \quad L^{p}_{loc}(\mathbb{R}^+ )\quad \mbox{as}
\quad j \rightarrow +\infty \quad \forall \,\, 1
<p<+\infty.\end{equation}
\end{prop}

\begin{proof}
According to \eqref{identity-2}, one has that $ u_{\varepsilon} \in L^{\infty}(\mathbb{R}^{+}; H^1(\mathbb{S})), $
and $ u_{\varepsilon} $ is uniformly bounded in $H^1(\mathbb{S}). $
While the first equation in \eqref{appr-equation-1} yields
\begin{equation*}\partial_t u_{\varepsilon}=\varepsilon \partial_x^2
u_{\varepsilon}-u_{\varepsilon} \partial_x
u_{\varepsilon}-\partial_x P_{\varepsilon}.\end{equation*} Thanks to
\eqref{formula-2.2}, \eqref{identity-2} and \eqref{e2.2}, together
with H\"{o}lder's inequality applied, we obtain
\begin{equation*}\label{bdd-p-infty-1}
\begin{split}
\|\partial_xP_{\varepsilon}\|_{L^{2}(\mathbb{S})}\leq
\|\partial_xP_{\varepsilon}\|_{L^{\infty}(\mathbb{S})} \leq
\frac{9}{4}\left(\mu_0^2+(\mu_0+\kappa)^2\right)
+\left(\frac{1}{4}+\frac{1}{2\pi^2}\right)\mu_1^2,
\end{split}
\end{equation*}
\begin{equation*}
\begin{split}
\|u_{\varepsilon} \partial_x u_{\varepsilon}\|_{L^{2}(\mathbb{S})}
&\leq \|u_{\varepsilon} \|_{L^{\infty}}\|\partial_x
u_{\varepsilon}\|_{L^{2}}\leq
\mu_1\sqrt{\mu_0^2+\frac{1}{4\pi^2}\mu_1^2} \quad \mbox{and}
\end{split}
\end{equation*}
\begin{equation*}
\begin{split}
\sqrt{\varepsilon} \|\partial_x^2 u_{\varepsilon}\|_{L^{2}([0, T]
\times \mathbb{S})} \leq \mu_1 \quad \mbox{for} \quad \forall \quad
T>0.
\end{split}
\end{equation*}
So, we have
\begin{equation*}\label{bdd-time-u}\|\partial_t
u_{\varepsilon}\|_{L^{2}([0, T] \times \mathbb{S})}^2 \leq C(\mu_0,
\mu_1, T)\end{equation*} with the constant $C(\mu_0, \mu_1, T)$
independent of $\varepsilon$, and consequently,
\begin{equation*}\partial_t u_{\varepsilon} \in L^2_{loc}(\mathbb{R}^{+};
L^2(\mathbb{S})).
\end{equation*}
From this, fixing $T>0$, we get for $0 \leq t, \, s \leq T$,
$$\|u_{\varepsilon}(t)-u_{\varepsilon}(s)\|_{L^{2}(\mathbb{S})}^2
=\int_{\mathbb{S}}\left(\int_s^{t} \partial_t u_{\varepsilon}(\tau,
x)\, d\tau\right)^2 \, dx \leq|t-s|^{\frac{1}{2}}  \|\partial_t
u_{\varepsilon}\|_{L^{2}([0, T] \times \mathbb{S})}^2.$$ Therefore,
applying Lemma \ref{lemma-compactness-1}, together with the
embedding theorem $H^1(\mathbb{S}) \hookrightarrow\hookrightarrow
C(\mathbb{S}) \hookrightarrow L^{2}(\mathbb{S})$, we deduce
\eqref{conti-u}.

On the other hand, we get from the second equation in
\eqref{identity-2} that
\begin{equation*}\label{P-2-identity-1-a}
\|\mu((\partial_xu_{\varepsilon_j})^2)\|_{L^{\infty}(\mathbb{R}^+)}
\leq \mu_1^2 \quad \mbox{and}
\end{equation*}
\begin{equation}\label{P-2-identity-1}
\frac{d}{dt} \mu((\partial_xu_{\varepsilon_j})^2) =-2\varepsilon
\|\partial_x^2u_{\varepsilon_j}\|_{L^2}^2,
\end{equation}
which together with \eqref{identity-2} once again implies
\begin{equation*}\label{P-2-identity-1-b}
\|\frac{d}{dt}
\mu((\partial_xu_{\varepsilon_j})^2)\|_{L^2(\mathbb{R}^+)} \leq
\mu_1^2.
\end{equation*}
Therefore, the standard Lions-Aubin's Lemma applied implies
\eqref{strong-p}. This completes the proof of Proposition
\ref{prop-compactness-1}.
\end{proof}

Now let $\mu_{t, x}(\lambda)$ be the Young measure associated with
$\{q_{\varepsilon}\}_{\varepsilon>0}$. The theory of the Young
measures (see Lemma 4.2 in \cite{xz}, also \cite{jmr}, \cite{young},
\cite{zhangz}, \cite{zhangz-3}, \cite{zhangz-4}, \cite{zhangz-2})
applied implies that, for any continuous function $f=f(\lambda)$
such that $f(\lambda)=o(|\lambda|^{r})$ and
$\partial_{\lambda}f(\lambda)=o(|\lambda|^{r-1})$ as
$|\lambda|\rightarrow +\infty$ and $r < 2$, and for any $\psi \in
L^{s}(\mathbb{S})$ with $\frac{1}{s}+\frac{r}{2}=1$, there holds
\begin{equation}\label{young-1-1}
\lim_{\varepsilon \rightarrow
0}\int_{\mathbb{S}}f(q_{\varepsilon}(t, x))\psi(x)
\,dx=\int_{\mathbb{S}}\overline{f(q)}\psi(x) \,dx
\end{equation}
uniformly in every compact subset of $\mathbb{R}^{+}$. Here
\begin{equation*}
\overline{f(q)}:=\int_{\mathbb{S}}f(x) \,d\mu_{t, x}(\lambda) \in
C([0, \infty); L^{r'/r}(\mathbb{S}))
\end{equation*}
with $r' \in (r, 2)$. Moreover, for all $T > 0$, there hold
\begin{equation*}
\lim_{\varepsilon \rightarrow
0}\int_{0}^{T}\int_{\mathbb{S}}g(q_{\varepsilon})\varphi
\,dxdt=\int_{0}^{T}\int_{\mathbb{S}}\overline{g(q)}\varphi \,dxdt
\end{equation*}
and
\begin{equation*}
\lambda \in L_{loc}^{\ell}(\mathbb{R}^{+}\times \mathbb{S}\times
\mathbb{R}, \, dt \otimes\,dx\otimes\,d\mu_{t, x}(\lambda))) \quad
\mbox{for all} \quad \ell <3,
\end{equation*}
where $g=g(t, x, \lambda)$ is a continuous function satisfying $g =
o(|\lambda|^{\ell})$ as $|\lambda|\rightarrow +\infty$  for some
$\ell < 3$, and  with $\frac{\ell}{3} + \frac{1}{m}< 1$. And also
\begin{equation}\label{u-q-re-1}
\lambda \in L^{\infty}\left(\mathbb{R}^{+}; L^2(\mathbb{S}\times
\mathbb{R}, \, \,dx\otimes\,d\mu_{t, x}(\lambda))\right) \quad
\mbox{and} \quad \overline{q}(t, x)=\partial_x u(t, x).
\end{equation}

We are in a position to study the structure of the Young measure
$\mu_{t, x}(\lambda)$.
\begin{lem}\label{lem-renormalize-1} Let $E=E(\lambda)\in W^{2, \infty}(\mathbb{R})$
 be a given convex function satisfying $E(\lambda)=O(|\lambda|)$ and the first derivative $DE(\lambda)=O(1)$
  for $|\lambda|\rightarrow +\infty$. Then there holds
\begin{equation}\label{q-equation-1}
\partial_t \overline{E(q)}+\partial_x (u \overline{E(q)})\leq \overline{qE(q)}
-\frac{1}{2}\overline{q^2DE(q)}+\overline{DE(q)}[2(u-\mu_0)(\mu_0+\kappa)-\frac{1}{2}\Pi_1]
\end{equation}
in the sense of distributions on $\mathbb{R}^{+} \times \mathbb{S}$.
\end{lem}
\begin{proof}
Multiplying the first equation in \eqref{one-derivative-1} by
$DE(q_{\varepsilon})$, we get
\begin{equation}\label{q-equation-2a}
\begin{split}
\partial_t E(q_{\varepsilon})+u_{\varepsilon}\partial_x
E(q_{\varepsilon}) &-\varepsilon
DE(q_{\varepsilon})\partial_x^2q_{\varepsilon}
+\frac{1}{2}DE(q_{\varepsilon})(q_{\varepsilon})^2\\
&
=DE(q_{\varepsilon})\left(2(u_{\varepsilon}-\mu_0)(\mu_0+\kappa)-\frac{1}{2}\mu(q_{\varepsilon}^2)\right),
\end{split}
\end{equation}
which implies
\begin{equation*}
\begin{split}
\partial_t E(q_{\varepsilon})&+\partial_x
(u_{\varepsilon}E(q_{\varepsilon}))=q_{\varepsilon}E(q_{\varepsilon})+\varepsilon
\partial_x(DE(q_{\varepsilon})
\partial_xq_{\varepsilon})-\varepsilon D^2E(q_{\varepsilon})
(\partial_xq_{\varepsilon})^2\\
&-\frac{1}{2}DE(q_{\varepsilon})(q_{\varepsilon})^2
+DE(q_{\varepsilon})\left(2(u_{\varepsilon}-\mu_0)(\mu_0+\kappa)-\frac{1}{2}\mu(q_{\varepsilon}^2)\right).
\end{split}
\end{equation*}
Noting that $\sqrt{\varepsilon}\partial_xq_{\varepsilon}$ is
uniformly bounded in $L^2(\mathbb{R}^{+} \times \mathbb{S})$
(according to \eqref{identity-2}), and taking the limit $\varepsilon
\rightarrow 0$, one obtains from Proposition
\ref{prop-compactness-1} and \eqref{u-q-re-1} that
\eqref{q-equation-1} holds.
\end{proof}
Taking $E(\lambda)=\lambda$ in \eqref{q-equation-2a} gives
\begin{equation*}
\begin{split}
\partial_t q_{\varepsilon}+\partial_x
(u_{\varepsilon}q_{\varepsilon})=\varepsilon
\partial_x^2q_{\varepsilon}+\frac{1}{2}(q_{\varepsilon})^2
+\left(2(u_{\varepsilon}-\mu_0)(\mu_0+\kappa)-\frac{1}{2}\mu(q_{\varepsilon}^2)\right).
\end{split}
\end{equation*}
Similar to the proof of Lemma \ref{lem-renormalize-1}, we may get
\begin{lem}\label{lem-q-2} There holds
\begin{equation}\label{q-equation-1-1-2}
\partial_t \overline{q}+u\partial_x \overline{q}=(\frac{1}{2}\overline{q^2}-\overline{q}^2)
+\left(2(u-\mu_0)(\mu_0+\kappa)-\frac{1}{2}\Pi_1\right)
\end{equation}
in the sense of distributions on $\mathbb{R}^{+} \times \mathbb{S}$.
\end{lem}

\begin{lem}\label{lem-renormalize-2} Let $E=E(\lambda)\in W^{2, \infty}(\mathbb{R})$
 be a given convex function satisfying $E(\lambda)=O(|\lambda|)$ and $DE(\lambda)=O(1)$
  for $|\lambda|\rightarrow +\infty$. Then there holds
\begin{equation}\label{q-equation-2}
\begin{split}
&\partial_t (\overline{E(q)}-E(\overline{q}))+\partial_x (u
(\overline{E(q)}-E(\overline{q})))\\
 &\leq
\int_{\mathbb{R}}\left(\lambda E(\lambda)-\frac{1}{2}DE(\lambda)
\lambda^2\right)\,d\mu_{t, x}(\lambda)+\frac{1}{2}DE(\overline{q})
(\overline{q})^2-\overline{q}E(\overline{q})\\
&\quad -\frac{1}{2}DE(\overline{q})(\overline{q^2}-(\overline{q})^2)
+(\overline{DE(q)}-DE(\overline{q}))\left(2(u-\mu_0)(\mu_0+\kappa)-\frac{1}{2}\Pi_1\right)
\end{split}
\end{equation}
in the sense of distributions on $\mathbb{R}^{+} \times \mathbb{S}$.
\end{lem}
\begin{proof} We first get from \eqref{q-equation-1-1-2} that
\begin{equation}\label{q-appr-1}
\partial_t {\overline{q}}+u\partial_x {\overline{q}}=
(\frac{1}{2}\overline{q^2}-\overline{q}^2)
+\left(2(u-\mu_0)(\mu_0+\kappa)-\frac{1}{2}\Pi_1\right)
\end{equation}
Taking the convolution of \eqref{q-appr-1} with the standard
Friedrichs mollifier, $j_{\delta}(x)$, one gets that
\begin{equation}\label{q-appr-1a}
\partial_t {\overline{q}}^{\delta}+u\partial_x {\overline{q}}^{\delta}=
j_{\delta} \ast \left((\frac{1}{2}\overline{q^2}-\overline{q}^2)
+\left(2(u-\mu_0)(\mu_0+\kappa)-\frac{1}{2}\Pi_1\right)\right)+r_{\delta},
\end{equation}
where ${\overline{q}}^{\delta}=j_{\delta}\ast \overline{q}$,\,
$r_{\delta}=u\partial_x {\overline{q}}^{\delta}-j_{\delta}
\ast(u\partial_x \overline{q})$. Multiplying \eqref{q-appr-1a} by
$DE({\overline{q}}^{\delta})$ gives rise to
\begin{equation}\label{q-appr-1b}
\begin{split}
&\partial_t E({\overline{q}}^{\delta})+\partial_x
(uE({\overline{q}}^{\delta}))=\overline{q}E({\overline{q}}^{\delta})\\
&\qquad +DE({\overline{q}}^{\delta})\left\{j_{\delta} \ast
\left((\frac{1}{2}\overline{q^2}-\overline{q}^2)
+\left(2(u-\mu_0)(\mu_0+\kappa)-\frac{1}{2}\Pi_1\right)\right)+r_{\delta}\right\}.
\end{split}
\end{equation}
Taking the limit $\delta \rightarrow 0^{+}$ in \eqref{q-appr-1b} and
using the fact that,
\begin{equation*}r_{\delta}\rightarrow 0 \quad \mbox{as}\quad \delta \rightarrow 0^{+}
 \quad \mbox{in}\quad L^1_{loc}(\mathbb{R}^{+}, L^{1}(\mathbb{S})),\end{equation*}
which follows from lemma II.1 of \cite{Di-Li}, one obtains that
\begin{equation}\label{q-appr-2}
\begin{split}
\partial_t E(\overline{q})&+\partial_x (u
E(\overline{q})\\
&=\overline{q}E(\overline{q})+DE(\overline{q})\left(\frac{1}{2}\overline{q^2}
-(\overline{q})^2
-\left(2(u-\mu_0)(\mu_0+\kappa)-\frac{1}{2}\Pi_1\right)\right).
\end{split}
\end{equation}
Subtracting \eqref{q-appr-2} from \eqref{q-equation-1} yields
\eqref{q-equation-2}.
\end{proof}

\begin{lem}
For each $R>0$,
\begin{equation}\label{initial-q-2}
\lim_{t\rightarrow 0+}\int_{\mathbb{S}}(\overline{Q_{R}^{\pm}(q)}(t,
x)-Q_{R}^{\pm}(\overline{q})(t, x))\, dx=0,
\end{equation}
where
\begin{equation*}\label{initial-q-2-2a}
Q_{R}(\lambda)=\begin{cases} \frac{1}{2}\lambda^2 \quad & if \,
|\lambda| \leq R,\\
R|\lambda|-\frac{1}{2}R^2 & if \,|\lambda| \geq R,
\end{cases}
\end{equation*}
and $Q_{R}^{+}(\lambda):={\bf{1}}_{\lambda \geq 0}Q_{R}(\lambda)$,
$Q_{R}^{-}(\lambda):={\bf{1}}_{\lambda \leq 0}Q_{R}(\lambda)$ for
$\lambda \in \mathbb{R}$, where ${\bf{1}}_{A}$ denotes the
characteristic function of the set $A$.
\end{lem}
\begin{proof}
Thanks to the definition of $Q_{R}(\lambda)$ and
$Q_{R}^{\pm}(\lambda)$, one may verify that $Q_{R}^{\pm}(\lambda)$
and $Q_{R}(\lambda)=Q_{R}^{+}(\lambda)+Q_{R}^{-}(\lambda)$ all
satisfy the assumptions on $E$ in Lemma \ref{lem-renormalize-2}, so
one can apply \eqref{q-equation-2} to all of them.

Note that $Q_{R}(\lambda)$ is a convex function, we get from
Jensen's inequality that
\begin{equation*}
\begin{split}
0 &\leq \overline{Q_{R}(q)}(t,
x)-Q_{R}(\overline{q})\\
&=\frac{1}{2}\left(\overline{q^2}-(\overline{q})^2\right)-\frac{1}{2}
\bigg(\int_{\mathbb{R}}(|\lambda|-R)^2{\bf 1}_{|\lambda|\geq
R}\,d\mu_{t, x}(\lambda)-(|\overline{q}|-R)^2{\bf
1}_{|\overline{q}|\geq R}\bigg).
\end{split}
\end{equation*}
While $(|\lambda|-R)^2{\bf 1}_{|\lambda|\geq R}$ is a convex
function, one gets that
\begin{equation*}
\int_{\mathbb{R}}(|\lambda|-R)^2{\bf 1}_{|\lambda|\geq R}\,d\mu_{t,
x}(\lambda)-(|\overline{q}|-R)^2{\bf 1}_{|\overline{q}|\geq R} \geq
0.
\end{equation*}
Hence,
\begin{equation}\label{inequality-Q-q-1}
0 \leq \overline{Q_{R}^{\pm}(q)}-Q_{R}^{\pm}(\overline{q}) \leq
\overline{Q_{R}(q)}-Q_{R}(\overline{q}) \leq
\frac{1}{2}(\overline{q^2}-(\overline{q})^2).
\end{equation}
On the other hand, thanks to the fact that $u \in C(\mathbb{R}^{+}
\times \mathbb{S})$ and \eqref{u-q-re-1}, we get for each test
function $\phi \in C^{\infty}(\mathbb{S})$
\begin{equation*}
\begin{split}
 \lim_{t \rightarrow 0^+} \int_{\mathbb{S}}q(t, x) \phi(x)\, dx
&=-\lim_{t \rightarrow 0^+} \int_{\mathbb{S}}u(t, x)
\partial_x \phi(x)\, dx\\
&=- \int_{\mathbb{S}}u_0(x)
\partial_x \phi(x)\,
dx =\int_{\mathbb{S}}q_0( x) \phi(x)\, dx.
\end{split}
\end{equation*}
From this, together with the fact that $q_{\varepsilon} (\in
C(\mathbb{R}^+; L^{2}(\mathbb{S}))\cap L^{\infty}(\mathbb{R}^+;
L^{2}(\mathbb{S})))$ is uniformly bounded with respect to
$\varepsilon>0$, we obtain that
\begin{equation*}
\overline{q}(t, x) \rightharpoonup q_0(x) =\partial_x u_0 \quad
\mbox{as} \quad t \rightarrow 0^+\quad \mbox{in} \quad
L^{2}(\mathbb{S}),
\end{equation*}
and so
\begin{equation*}
\lim_{t \rightarrow 0^+} \int_{\mathbb{S}}(\overline{q}(t, x))^2\,
dx \geq \int_{\mathbb{S}}({q}_0(x))^2\, dx.
\end{equation*}
While the energy estimate \eqref{identity-2} together with
\eqref{young-1-1} implies that
\begin{equation*}
\lim_{t \rightarrow 0^+} \int_{\mathbb{S}}(\overline{q}(t, x))^2\,
dx \leq \lim_{t \rightarrow 0^+} \int_{\mathbb{S}}\overline{q(t,
x)^2}\, dx \leq \int_{\mathbb{S}}({q}_0(x))^2\, dx.
\end{equation*}
Hence, we have
\begin{equation*}
\lim_{t \rightarrow 0^+} \int_{\mathbb{S}}(\overline{q}(t, x))^2\,
dx =\lim_{t \rightarrow 0^+} \int_{\mathbb{S}}\overline{q(t, x)^2}\,
dx = \int_{\mathbb{S}}({q}_0(x))^2\, dx,
\end{equation*}
which along with \eqref{inequality-Q-q-1} implies
\eqref{initial-q-2}.
\end{proof}
We are in a position to prove that the Young measure $\mu_{t,
x}(\lambda)$ is a Dirac measure.
\begin{lem}\label{lemma-Young-measure-1}
Let $\mu_{t, x}(\lambda)$ be the Young measure associated with
$\{q_{\varepsilon}\}_{\varepsilon>0}$. Then
\begin{equation}\label{Young-measure-1}\mu_{t, x}(\lambda)
=\delta_{\overline{q}(t, x)}(\lambda) \quad \forall \quad a. e.
 \quad (t, x) \in \mathbb{R}^{+} \times \mathbb{S}.
 \end{equation}
\end{lem}
\begin{proof}
 We first apply \eqref{q-equation-2} to
$E(\lambda)=Q_{R}^{+}(\lambda)$ to obtain
\begin{equation}\label{q-r-equation-3}
\begin{split}
&\partial_t
(\overline{Q_{R}^{+}(q)}-Q_{R}^{+}(\overline{q}))+\partial_x \left(u
\left(\overline{Q_{R}^{+}(q)}-Q_{R}^{+}(\overline{q})\right)\right)\\
 &\leq \frac{R}{2}\left(
\int_{\mathbb{R}}\lambda(\lambda-R) {\bf 1}_{\lambda \geq
R}\,d\mu_{t, x}(\lambda)-\overline{q}(\overline{q}-R) {\bf
1}_{\overline{q} \geq
R}\right)\\
&
-\frac{1}{2}DQ_{R}^{+}\left(\overline{q})(\overline{q^2}-(\overline{q})^2\right)
+(\overline{DQ_{R}^{+}(q)}-DQ_{R}^{+}(\overline{q}))\left(2(u-\mu_0)(\mu_0+\kappa)
-\frac{1}{2}\Pi_1\right).
\end{split}
\end{equation}
Note that both $\overline{q}(t, x)$ and $q_{\varepsilon}(t, x)$ are
bounded above by $\frac{2}{t}+C$ with $C=L_0$ in
\eqref{entrop-con-1}. Thus, $\mbox{Supp} \, \mu_{t, x}(\cdot)
\subset (-\infty, \frac{2}{t}+C)$. Therefore, for $R\geq
\frac{2}{t}+C$, i. e., $t \geq \frac{2}{R-C}$ ( for $R>C$), one gets
from \eqref{q-r-equation-3} that
\begin{equation*}
\begin{split}
&\partial_t
(\overline{Q_{R}^{+}(q)}-Q_{R}^{+}(\overline{q}))+\partial_x \left(u
\left(\overline{Q_{R}^{+}(q)}-Q_{R}^{+}(\overline{q})\right)\right)\\
& \leq
(\overline{DQ_{R}^{+}(q)}-DQ_{R}^{+}(\overline{q}))\left(2(u-\mu_0)(\mu_0+\kappa)
-\frac{1}{2}\Pi_1\right),
\end{split}
\end{equation*}
which implies for $t \geq \frac{2}{R-C}$ that
\begin{equation}\label{q-r-equation-4}
\begin{split}
&\int_{\mathbb{S}}(\overline{Q_{R}^{+}(q)}-Q_{R}^{+}(\overline{q}))(t,
x) \, dx\leq
\int_{\mathbb{S}}(\overline{Q_{R}^{+}(q)}-Q_{R}^{+}(\overline{q}))(\frac{2}{R-C},
x) \, dx\\
&\qquad +\int_{\frac{2}{R-C}}^{t}
\int_{\mathbb{S}}(\overline{DQ_{R}^{+}(q)}-DQ_{R}^{+}(\overline{q}))
\left(2(u-\mu_0)(\mu_0+\kappa)-\frac{1}{2}\Pi_1\right) \, dx ds.
\end{split}
\end{equation}
For any $f$, define $f_{+}:=\max\{f, \, 0\}$, $f_{-}:=\min\{f, \,
0\}$. Using this notation, together with the definition of
$Q_{R}^{+}(q)$, one has
\begin{equation*}
\begin{split}
\overline{Q_{R}^{+}(q)}-Q_{R}^{+}(\overline{q})&=\frac{1}{2}(\overline{q_{+}^2}-(\overline{q}_+)^2)
-\frac{1}{2}\left\{ \int_{\mathbb{R}}(\lambda-R)^2 {\bf 1}_{\lambda
\geq R}\,d\mu_{t, x}(\lambda)-(\overline{q}-R)^2 {\bf
1}_{\overline{q} \geq
R}\right\}\\
&=\frac{1}{2}(\overline{q_{+}^2}-(\overline{q}_+)^2)
\end{split}
\end{equation*}
and
\begin{equation*}
\begin{split}
&\overline{DQ_{R}^{+}(q)}-DQ_{R}^{+}(\overline{q})=(\overline{q_{+}}-\overline{q}_+)
-\left\{ \int_{\mathbb{R}}(\lambda-R) {\bf 1}_{\lambda \geq
R}\,d\mu_{t, x}(\lambda)-(\overline{q}-R){\bf 1}_{\overline{q} \geq
R}\right\},
\end{split}
\end{equation*}
which applied to \eqref{q-r-equation-4} gives rise to
\begin{equation*}\label{q-r-equation-5}
\begin{split}
&\int_{\mathbb{S}}(\overline{q_{+}^2}-(\overline{q}_+)^2)(t, x) \,
dx\leq 2\int_{\frac{2}{R-C}}^{t}
\int_{\mathbb{S}}(\overline{q_{+}}-\overline{q}_+)\left(2(u-\mu_0)(\mu_0+\kappa)
-\frac{1}{2}\Pi_1\right)
\, dx ds
\\
&\quad
+\int_{\mathbb{S}}(\overline{q_{+}^2}-(\overline{q}_+)^2)(\frac{2}{R-C},
x) \, dx-2\int_{\frac{2}{R-C}}^{t}
\int_{\mathbb{S}}\left(2(u-\mu_0)(\mu_0+\kappa)-\frac{1}{2}\Pi_1\right)\\
 &\qquad \qquad \qquad \qquad \times\left(\int_{\mathbb{R}}(\lambda-R)
{\bf 1}_{\lambda \geq R}\,d\mu_{t, x}(\lambda)-(\overline{q}-R){\bf
1}_{\overline{q} \geq R}\right)\, dx ds.
\end{split}
\end{equation*}
Taking the limit $R\rightarrow +\infty$ and using \eqref{u-q-re-1},
\eqref{initial-q-2}, and the Lebesgue dominated convergence theorem,
we conclude that for all $t
> 0$
\begin{equation}\label{q-r-equation-6}
\begin{split}
&\int_{\mathbb{S}}(\overline{q_{+}^2}-(\overline{q}_+)^2)(t, x) \,
dx\leq 2\int_{0}^{t}
\int_{\mathbb{S}}(\overline{q_{+}}-\overline{q}_+)\left(2(u-\mu_0)(\mu_0+\kappa)-\frac{1}{2}\Pi_1\right)
\, dx ds.
\end{split}
\end{equation}
Since
$\overline{q^2}-(\overline{q})^2=(\overline{q_{+}^2}-(\overline{q}_+)^2)
+(\overline{q_{-}^2}-(\overline{q}_-)^2)$, it remains to estimate
the part associated with $(\overline{q_{-}^2}-(\overline{q}_-)^2)$,
which may be approximated by
$\overline{Q_{R}^{-}(q)}-Q_{R}^{-}(\overline{q})$ as $R$ goes to
$+\infty$.

Indeed, we apply \eqref{q-equation-2} to
$E(\lambda)=Q_{R}^{-}(\lambda)$ to obtain
\begin{equation*}
\begin{split}
&\partial_t
(\overline{Q_{R}^{-}(q)}-Q_{R}^{-}(\overline{q}))+\partial_x (u
(\overline{Q_{R}^{-}(q)}-Q_{R}^{+}(\overline{q})))\\
 &\leq -\frac{R}{2}\bigl\{
\int_{\mathbb{R}}\lambda(\lambda+R) {\bf 1}_{\lambda \leq
-R}\,d\mu_{t, x}(\lambda)-\overline{q}(\overline{q}+R) {\bf
1}_{\overline{q} \leq
-R}\bigr\}\\
&
-\frac{1}{2}DQ_{R}^{-}(\overline{q})(\overline{q^2}-(\overline{q})^2)
+(\overline{DQ_{R}^{-}(q)}-DQ_{R}^{-}(\overline{q}))\left(2(u-\mu_0)(\mu_0+\kappa)
-\frac{1}{2}\Pi_1\right).
\end{split}
\end{equation*}
Hence, integrating this inequality over $[0, t) \times \mathbb{S}$
and using \eqref{initial-q-2}, we get that
\begin{equation}\label{q-r-equation-7}
\begin{split}
&\int_{\mathbb{S}}(\overline{Q_{R}^{-}(q)}-Q_{R}^{-}(\overline{q}))(t,
x) \, dx\leq  \frac{R}{2}\int_{0}^{t}
\int_{\mathbb{S}}(\overline{q^2}-(\overline{q})^2)\,dxds\\
 &\qquad -\frac{R}{2}\int_{0}^{t}
\int_{\mathbb{S}}\left\{ \int_{\mathbb{R}}\lambda(\lambda+R) {\bf
1}_{\lambda \leq -R}\,d\mu_{t,
x}(\lambda)-\overline{q}(\overline{q}+R) {\bf 1}_{\overline{q} \leq
-R}\right\}\,dxds\\
&\qquad +\int_{0}^{t} \int_{\mathbb{S}}(\overline{DQ_{R}^{-}(q)}
-DQ_{R}^{-}(\overline{q}))\left(2(u-\mu_0)(\mu_0+\kappa)-\frac{1}{2}\Pi_1\right)\,dxds.
\end{split}
\end{equation}
While a direct computation yields
\begin{equation*}
\begin{split}
&\overline{Q_{R}^{-}(q)}-Q_{R}^{-}(\overline{q})\\
&=\frac{1}{2}(\overline{q_{-}^2}-(\overline{q}_-)^2)
-\frac{1}{2}\bigl\{ \int_{\mathbb{R}}(\lambda+R)^2 {\bf 1}_{\lambda
\leq -R}\,d\mu_{t, x}(\lambda)-(\overline{q}+R)^2 {\bf
1}_{\overline{q} \leq -R}\bigr\},
\end{split}
\end{equation*}
which together with \eqref{q-r-equation-6} and
\eqref{q-r-equation-7} leads to
\begin{equation}\label{q-r-equation-8}
\begin{split}
&\int_{\mathbb{S}}\left(\frac{1}{2}(\overline{q_{+}^2}-(\overline{q}_+)^2)
+\overline{Q_{R}^{-}(q)}-Q_{R}^{-}(\overline{q})\right)(t,
x) \, dx\\
 &\leq R\int_{0}^{t} \int_{\mathbb{S}}\left(\frac{1}{2}(\overline{q_{+}^2}-(\overline{q}_+)^2)
+\overline{Q_{R}^{-}(q)}-Q_{R}^{-}(\overline{q})\right)(s, x) \, dxds\\
&+\frac{R}{2}\int_{0}^{t} \int_{\mathbb{S}}\left\{
\int_{\mathbb{R}}R(\lambda+R) {\bf 1}_{\lambda \leq -R}\,d\mu_{t,
x}(\lambda)-R(\overline{q}+R) {\bf 1}_{\overline{q} \leq
-R}\right\}\,dxds\\
&+\int_{0}^{t}
\int_{\mathbb{S}}\left(\overline{DQ_{R}^{-}(q)}-DQ_{R}^{-}(\overline{q})+
\overline{q_{+}}-\overline{q}_+\right)\left(2(u-\mu_0)(\mu_0+\kappa)-\frac{1}{2}\Pi_1\right)\,dxds.
\end{split}
\end{equation}
Note that
\begin{equation*}
\begin{split}
0  \leq
\overline{DQ_{R}^{-}(q)}-DQ_{R}^{-}(\overline{q})&+\overline{q_{+}}-\overline{q}_{+}\\
& =-\left( \int_{\mathbb{R}}(\lambda+R) {\bf 1}_{\lambda \leq -R}\,
d\mu_{t, x}(\lambda)-(\overline{q}+R){\bf 1}_{\overline{q} \leq
-R}\right).
\end{split}
\end{equation*}
Let $L>0$ (for example, taking $L=L_0^2$ in \eqref{entrop-con-1}) be
a constant such that $
  \|2(u-\mu_0)(\mu_0+\kappa)-\frac{1}{2}\mu(q_{\varepsilon}^2)\|_{L^{\infty}} \leq \frac{L}{2}
$ (see \eqref{inequality-1-3}). Then,
\begin{equation*}
\begin{split}
&\int_{0}^{t}
\int_{\mathbb{S}}\left(\overline{DQ_{R}^{-}(q)}-DQ_{R}^{-}(\overline{q})+
\overline{q_{+}}-\overline{q}_+\right)\left(2(u-\mu_0)(\mu_0+\kappa)-\frac{1}{2}\Pi_1\right)\,dxds\\
&\leq \frac{R}{2}\int_{0}^{t} \int_{\mathbb{S}}\left(
\int_{\mathbb{R}}\frac{L}{R}(\lambda+R) {\bf 1}_{\lambda \leq
-R}\,d\mu_{t, x}(\lambda)-\frac{L}{R}(\overline{q}+R) {\bf
1}_{\overline{q} \leq -R}\right)\,dxds.
\end{split}
\end{equation*}
Therefore, for $R \geq \sqrt{L}$, we get from the fact that
$(\lambda+R) {\bf 1}_{\lambda \leq -R}$ is a concave function that
\begin{equation}\label{q-r-equation-9}
\begin{split}
&\frac{R}{2}\int_{0}^{t} \int_{\mathbb{S}}\left(
\int_{\mathbb{R}}R(\lambda+R) {\bf 1}_{\lambda \leq -R}\,d\mu_{t,
x}(\lambda)-R(\overline{q}+R) {\bf
1}_{\overline{q} \leq -R}\right)\,dxds\\
 &+\int_{0}^{t}
\int_{\mathbb{S}}\left(\overline{DQ_{R}^{-}(q)}-DQ_{R}^{-}(\overline{q})+
\overline{q_{+}}-\overline{q}_+\right)\left(2(u-\mu_0)(\mu_0+\kappa)-\frac{1}{2}\Pi_1\right)\,dxds\\
&\leq \frac{R}{2}\int_{0}^{t} \int_{\mathbb{S}}\left(
\int_{\mathbb{R}}(R-\frac{L}{R})(\lambda+R) {\bf 1}_{\lambda \leq
-R}\,d\mu_{t, x}(\lambda) -(R-\frac{L}{R})(\overline{q}+R) {\bf
1}_{\overline{q} \leq -R}\right)\,dxds\\& \leq 0.
\end{split}
\end{equation}
It follows from \eqref{q-r-equation-8}, \eqref{q-r-equation-9} and
Gronwall's inequality that
\begin{equation}\label{q-r-equation-10}
\begin{split}
\int_{\mathbb{S}}\left(\frac{1}{2}(\overline{q_{+}^2}-(\overline{q}_+)^2)
+\overline{Q_{R}^{-}(q)}-Q_{R}^{-}(\overline{q})\right)(t, x) \,
dx=0, \quad \forall \quad t \geq 0.
\end{split}
\end{equation}
Thus, by Fatou's lemma, one can take the limit as $R \rightarrow
+\infty$ in \eqref{q-r-equation-10} to conclude that
\begin{equation*}\label{q-r-equation-11}
\begin{split}
\int_{\mathbb{S}}(\overline{q^2}-(\overline{q})^2)(t, x) \, dx\leq
0, \quad \forall \quad t \geq 0.
\end{split}
\end{equation*}
From this, together with the fact $(\overline{q})^2 \leq
\overline{q^2}$, we get
\begin{equation*}\label{q-r-equation-12}
\begin{split}
\int_{\mathbb{S}}\overline{q^2}(t, x) \,
dx=\int_{\mathbb{S}}(\overline{q})^2(t, x) \, dx, \quad \forall
\quad t \geq 0,
\end{split}
\end{equation*}
which implies \eqref{Young-measure-1}.
\end{proof}

\subsection{Proof of Theorem \ref{thm-main-1}} \

\begin{proof}[Proof of  Theorem \ref{thm-main-1}]

With all the preparations given in the previous subsection, we are
in a position to conclude the proof of the theorem. Let $u(t, x)$ be
the limit of the viscous approximate solutions $u_{\varepsilon}(t,
x)$ as $\varepsilon \rightarrow 0^{+}$. It then follows from
Propositions \ref{thm-regu-1}, \ref{prop-entrop-con} and
\ref{prop-compactness-1} that $u(t, x) \in C(\mathbb{R}^{+} \times
\mathbb{S})\cap L^{\infty}(\mathbb{R}^{+}, H^{1}(\mathbb{S}))$,
$\Pi(t) \in L^{\infty}(\mathbb{R}^{+})$ and \eqref{energy-ine-1}
\eqref{entropy-time-1} hold.

Now we claim that
\begin{equation}\label{q-r-equation-14-a}q_{\varepsilon}=\partial_{x}u_{\varepsilon} \rightarrow q=\partial_{x}u
\quad \mbox{as} \quad \varepsilon \rightarrow 0^{+} \quad \mbox{in}
\quad L^2_{loc}(\mathbb{R}^{+} \times \mathbb{S}).
\end{equation}
Indeed, it follows from \eqref{u-q-re-1} and Lemma
\ref{lemma-Young-measure-1} that there exists a subsequence of
$\{u_{\varepsilon}(t, x)\}$, still denoted by itself, such that
\begin{equation*}\label{q-r-equation-15}q_{\varepsilon}=\partial_{x}u_{\varepsilon} \rightarrow q=\partial_{x}u
 \quad \mbox{in}
\quad L^{p_1}_{loc}(\mathbb{R}^{+}, L^{p_2}(\mathbb{S})) \quad
\forall \quad p_1 < \infty, \quad p_2 < 2.
\end{equation*}
This together with Proposition \ref{prop-high-inte} and a standard
interpolation theorem applied implies
\begin{equation}\label{q-r-equation-16}q_{\varepsilon}=\partial_{x}u_{\varepsilon} \rightarrow q=\partial_{x}u
 \quad \mbox{in}
\quad L^{p}_{loc}(\mathbb{R}^{+} \times \mathbb{S}) \quad \forall
\quad p < 3,
\end{equation} which gives
\eqref{q-r-equation-14-a}.

Thus, we get from \eqref{P-2-identity-1} and \eqref{identity-2} that
$ \Pi(t)=\mu((\partial_xu)^2)$.

Taking $\varepsilon \rightarrow 0^{+}$ in \eqref{appr-equation-1},
one finds from \eqref{q-r-equation-14-a} and Proposition
\ref{prop-compactness-1} that $u$ is an admissible weak solution to
\eqref{e1.4}. It then follows from \eqref{q-r-equation-16}
that $\partial_x u \in L^{p}_{loc}(\mathbb{R}^{+} \times
\mathbb{S})$ for any $1 \leq p < 3$. Hence the local space-time
higher integrability estimate \eqref{high-integrability-1-1} holds.
This completes the proof of Theorem \ref{thm-main-1}.
\end{proof}

\bigskip

\noindent {\bf Acknowledgments.} The work of Gui is partially
supported by the NSF of China under the grant 11001111, and the
Jiangsu University grants 10JDG141 and 10JDG157. The work of Liu is
partially supported by the NSF grant DMS-0906099 and the NHARP grant
003599-0001-2009.


\begin{thebibliography}{50}

\bibitem{acdm} {\small \textsc{Mark S. Alber, Roberto Camassa, Darryl D. Holm and Jerrold E.
Marsden}, On the Link between Umbilic Geodesics and Soliton
Solutions of Nonlinear PDEs, {\it Proc. R. Soc. Lond. A}, {\bf 450}
(1995), 677-692.}

\bibitem {BrCo1}{\small \textsc{A. Bressan and A. Constantin,} Global
conservative solutions of the Camassa-Holm equation, { \it Arch.
Ration. Mech. Anal.,} \textbf{183} (2007), 215-239. }

\bibitem{BrCo2}{\small \textsc {A. Bressan and A. Constantin},
Global dissipative solutions of the Camassa-Holm equation, {\it
Anal. Appl.}, {\bf 5} (2007), 1-27.}




\bibitem{but} {\small \textsc{G. Buttazo, M. Giaquina, S. Hildebrandt},  One-Dimensional
Variational Problems: An Introduction, Clarendon Press, Oxford, 1998.}



\bibitem{cam} {\small \textsc{R. Camassa and D. D. Holm}, An integrable shallow water
equation with peaked solitons, {\it Phys. Rev. Lett.}, {\bf 71}
(1993), 1661-1664.}





\bibitem{CaHoTi}{\small \textsc{C. S. Cao, D. D. Holm and E. S.
Titi}, Traveling wave solutions for a class of one-dimensional
nonlinear shallow water wave models, \textit{J. Dynam. Differential
Equations}, \textbf{16} (2004), 167-178.}

\bibitem{chk} {\small \textsc{G. M. Coclite, H. Holden, and K. H. Karlsen}, Global weak solutions
to a generalized hyperelastic-rod wave equation, {\it SIAM J. Math.
Anal.}, {\bf 37} (2006), 1044-1069.}


\bibitem{con1} {\small \textsc{A. Constantin},  On the Cauchy problem for the periodic
Camassa-Holm equation, {\it J. Differential Equations}, {\bf 141}
(1997), 218-235.}

\bibitem{con2} {\small \textsc{A. Constantin},  On the Blow-up of solutions of a periodic shallow
water equation, {\it J. Nonlinear Sci.}, {\bf 10} (2000), 391-399.}


\bibitem{con4} {\small \textsc{A. Constantin and J. Escher},  On the blow-up rate and
the blow-up set of breaking waves for a shallow water equation, {\it
Math. Z.}, {\bf 233} (2000), 75-91.}

\bibitem{con5} {\small \textsc{A. Constantin and J. Escher},  Wave breaking for nonlinear
nonlocal shallow water equations, {\it Acta Math.}, {\bf 181}
(1998), 229-243.}

\bibitem{CoJo}{\small \textsc{A. Constantin and R. S. Johnson},
Propagation of very long water waves, with vorticity, over
variable depth, with applications to tsunamis, {\it Fluid Dynam.
Res.}, \textbf{40}(2008), 175-211.}


\bibitem{conl2} {\small \textsc{A. Constantin and D. Lannes}, The hydrodynamical relevance  of the
Camassa-Holm and Degasperis-Procesi equations, {\it Arch. Ration.
Mech. Anal.,} {\bf 192} (2009), 165-186.}

\bibitem{con6} {\small \textsc{A. Constantin and H. P. McKean},  A shallow water equation on
 the circle, {\it Comm. Pure Appl. Math.}, {\bf 52} (1999), 949-982.}

\bibitem{con-m} {\small \textsc{A. Constantin and L. Molinet}, Obtital stability of solitary waves
for a shallow water equation, {\it Phys. D}, {\bf 157}(2001), 75-89.}

\bibitem{con-s} {\small \textsc{A. Constantin and W. A. Strauss}, Stability of peakons, {\it Comm.
Pure Appl. Math.}, {\bf 53}(2000), 603-610.}


\bibitem{deg2} {\small \textsc{A. Degasperis, D. D. Holm, and A. N. W. Hone}, Integrable
and non-integrable equations with peakons, Nonlinear physics: theory
and experiment, \textbf{II} (Gallipoli, 2002), 37, World Sci. Publ.,
River Edge, NJ, 2003.}

\bibitem{La}{\small \textsc{M. Lakshmanan}, Integrable nonlinear
wave equations and possible connections to tsunami dynamics, in
Tsunami and nonlinear waves, pp. 31-49, Springer, Berlin, 2007.}

\bibitem{Di-Mo} {\small \textsc{K. E. Dika and L. Molinet}, Stability of multi antipeakon-peakons profile, {\it Disc. Cont. Dyn. Sys-Ser.B}, {\bf 12}
(2009), 561-577.}

\bibitem{Di-Mo2} {\small \textsc{K. E. Dika and L. Molinet}, Stability of multipeakons, {\it Ann. I. H. Poincar\'e}, {\bf 26}
(2009), 1517-1532.}



\bibitem{Di-Li} {\small \textsc{R. J. DiPerna and P. L. Lions},  Ordinary differential equations,
transport theory and Sobolev spaces, {\it Invent. Math.}, {\bf 98}
(1989), 511-547.}







\bibitem{flq} {\small \textsc{Y. Fu, Y. Liu and C. Qu}, On the blow-up structure for the
generalized periodic Camassa-Holm and Degasperis-Procesi equations,
preprint.}

\bibitem{fuc} {\small \textsc{B. Fuchssteiner and A. S. Fokas}, Symplectic structures,
their B\"{a}cklund transformations and hereditary symmetries, {\it
Phys. D}, {\bf 4} (1981/1982), 47-66.}

\bibitem{G-L-iumj} {\small \textsc{G. Gui, Y. Liu, and L. Tian},  Global existence and blow-up phenomena for  the peakon b-family of
equations, {\it Indiana University Mathematics Journal}, {\bf
57}(3) (2008), 1209-1234.}




\bibitem{hun1} {\small \textsc{J. K. Hunter and R. Saxton},
Dynamics of director fields, {\it SIAM J. Appl. Math.}, {\bf 51}
(1991), 1498-1521.}

\bibitem{hun2} {\small \textsc{J. K. Hunter and Y. Zheng}, On a completely integrable hyperbolic
variational equation, {\it Physica D}, {\bf 79} (1994), 361-386.}

\bibitem{jmr} {\small \textsc{J. L. Joly, G. M\'{e}tvier, and J. Rauch}, Focusing at a point and
absorption of nonlinear oscillations, {\it  Trans. Amer. Math.
Soc.}, {\bf 347} (1995), 3921-3969.}

\bibitem{kat1} {\small \textsc{T. Kato},  On the Korteweg-de Vries equation, {\it Manuscripta Math.},
{\bf 28} (1979), 89-99.}


\bibitem{kat} {\small \textsc{T. Kato and G. Ponce}, Communtator estimates and the
Euler and Navier-Stokes equations, {\it Comm. Pure Appl. Math.},
{\bf 41} (1988), 891-907.}

\bibitem{khe} {\small \textsc{B. Khesin, J. Lenells, and G. Misiolek},  Generalized Hunter-Saxton
equation and the geometry of the group of circle diffeomorphisms,
{\it Math. Ann.}, {\bf 342} (2008), 617-656.}


\bibitem{kou} {\small \textsc{S. Kouranbaeva}, The Camassa-Holm equation as a geodesic
flow on the diffeomorphism group, {\it J. Math. Phys.},  {\bf 40}
(1999), 857-868.}

\bibitem{len1} {\small \textsc{J. Lenells}, The
Hunter-Saxton equation describes the geodesic flow on a sphere, {\it
J. Geom. Phys.},  {\bf 57} (2007), 2049-2064.}

\bibitem{len2} {\small \textsc{J. Lenells, G. Misiolek, and F. Ti\u{g}lay}, Integrable evolution equations on
spaces of tensor densities and their peakon solutions, {\it Comm.
Math. Phys.}, {\bf 299} (2010), 129-161.}



\bibitem{mis1} {\small \textsc{G. Misiolek}, Classical solutions of the periodic Camassa-Holm
equation, {\it Geom. Funct. Anal.}, {\bf 12} (2002), 1080-1104.}

\bibitem{mis2} {\small \textsc{G. Misiolek}, A shallow water equation as a geodesic
flow on the Bott-Virasoro group, {\it J. Geom. Phys.}, {\bf 24} (1998), 203-208.}

\bibitem{lions} {\small \textsc{P.-L. Lions}, Mathematical topics in fluid mechanics,
{\it Vol. 2. Compressible models}, Oxford Lecture Series in
Mathematics and Its Applications, 10. Clarendon, Oxford University
Press, New York, 1998.}

\bibitem{GR} {\small \textsc{G. Rodriguez-Blanco}, On the Cauchy problem for
the Camassa-Holm equation, {\it Nonlinear Anal.,} {\bf 46} (2001),
309-327.}

\bibitem{simon} {\small \textsc{J. Simon}, Compact sets in the space $L^p((0, T), B)$,  {\it Ann. Mat. Pura.
Appl.}, {\bf 146} (1987), 65-96.}

\bibitem{wh} {\small \textsc{G. B. Whitham}, Linear and Nolinear
Waves, John Wiley \& Sons, New York, 1974.}

\bibitem{xin} {\small \textsc{Z. Xin}, Theory of viscous conservation laws, {\it Some recent topics
in conservation laws}, 141-193, L. Hsiao and Z. Xin, editors.
Studies in Advanced Mathematics, 15. American Mathematical
Society/International Press, 1999.}

\bibitem{xz} {\small \textsc{Z. Xin and P. Zhang}, On the weak solutions to a shallow water
equation, {\it Comm. Pure Appl. Math.}, {\bf 53} (2000), 1411-1433.}



\bibitem{young} {\small \textsc{L. C. Young}, Lectures on the calculus of variations and
optimal control theory, W. B. Saunders, Philadelphia-London-Toronto,
1969.}

\bibitem{zhangz} {\small \textsc{P. Zhang and Y. Zheng}, On
oscillations of an asymptotic equation of a nonlinear variational
wave equation, {\it Asymptotic Anal.}, {\bf 18} (1998), 307-327.}

\bibitem{zhangz-3} {\small \textsc{P. Zhang and Y. Zheng}, On the existence and uniqueness
of solutions to an asymptotic equation of a variational wave
equation, {\it Acta Math. Sin. (Engl. Ser.)}, {\bf 15} (1) (1999),
115-130.}

\bibitem{zhangz-4} {\small \textsc{P. Zhang and Y. Zheng}, Existence and uniqueness of solutions
of an asymptotic equation arising from a variational wave equation
with general data, {\it Arch. Ration. Mech. Anal.},  {\bf 155} (1)
(2000), 49-83.}

\bibitem{zhangz-2} {\small \textsc{P. Zhang and Y. Zheng}, Rarefactive solutions to a nonlinear
variational wave equation, {\it Comm. Partial Differential
Equations}, {\bf 26} (2001), 381-420.}




\end{thebibliography}
\end{document}